\documentclass[11pt]{article}

\usepackage{epsfig}
\usepackage{amsmath}
\usepackage{amsthm}
\usepackage{url}
\usepackage{comment}
\usepackage{fullpage}
\usepackage{thm-restate}
\usepackage{datetime}
\usepackage{lineno}

\usepackage{hyperref}

\graphicspath{{Figures/}}

\newtheorem{theorem}{Theorem}

\newtheorem{corollary}[theorem]{Corollary}
\newtheorem{proposition}[theorem]{Proposition}
\newtheorem{observation}[theorem]{Observation}
\newtheorem{claim}[theorem]{Claim}

\newcommand{\dcm}{\mathbf{DCM}}
\newcommand{\db}{\mathrm{DB}}
\newcommand{\zsc}{\mathrm{DBD}}
\newcommand{\zsl}{\mathrm{DBDL}}
\newcommand{\edb}{\mathrm{EDB}}
\newcommand{\edbl}{\mathrm{EDBL}}

\newcommand{\zr}{\mathrm{R}}
\newcommand{\zbr}{\mathbf{R}}
\newcommand{\zbu}{\mathbf{U}}
\newcommand{\zbl}{\mathbf{L}}
\newcommand{\zlhs}{\mathrm{LHS}}
\newcommand{\zrhs}{\mathrm{RHS}}

\begin{document}

\title{Disjoint compatibility graph \\ of non-crossing matchings \\ of points in convex position}
\author{Oswin Aichholzer\thanks{
Institute for Software Technology,
Technische Universit\"at Graz. 
Inffeldgasse 16b/II, A-8010 Graz, Austria.
E-mail address \href{mailto://oaich@ist.tugraz.at}{\texttt{oaich@ist.tugraz.at}}.}
\and
Andrei Asinowski\thanks{
Institute of Computer Science, 
Freie Universit\"at Berlin.
Takustra{\ss}e 9, 14195 Berlin, Germany.
E-mail address \href{mailto://asinowski@mi.fu-berlin.de}{\texttt{asinowski@mi.fu-berlin.de}}.}
\and
Tillmann Miltzow\thanks{
Institute of Computer Science, 
Freie Universit\"at Berlin.
Takustra{\ss}e 9, 14195 Berlin, Germany.
E-mail address \href{mailto://miltzow@mi.fu-berlin.de}{\texttt{miltzow@mi.fu-berlin.de}}.} }
\date{\today}

\maketitle

\begin{abstract}
Let $X_{2k}$ be a set of $2k$ labeled points in convex position in the plane.
We consider geometric non-intersecting straight-line perfect matchings of $X_{2k}$.
Two such matchings, $M$ and $M'$, are \textit{disjoint compatible}
if they do not have common edges, and no edge of $M$ crosses an edge of $M'$.
Denote by $\dcm_k$ the graph whose vertices correspond to such matchings,
and two vertices are adjacent if and only if
the corresponding matchings are disjoint compatible.
We show that for each $k \geq 9$, the connected components of $\dcm_k$
form exactly three isomorphism classes -- namely,
there is a certain number of isomorphic \textit{small} components,
a certain number of isomorphic \textit{medium} components,
and one \textit{big} component.
The number and the structure of small and medium components is determined precisely.

\smallskip

\end{abstract}

\section{Introduction}\label{sec:intro}
\subsection{Basic definitions and main results}\label{sec:results}
Let $k$ be a natural number,
and let $X_{2k}$ be a set of $2k$ points in convex position in the plane,
labeled circularly (say, clockwise) by $P_1, P_2, \dots, P_{2k}$
(in figures, we label them just by $1, 2, \dots, 2k$).
We consider geometric \textbf{perfect} matchings of $X_{2k}$ realized by \textbf{non-crossing straight segments}.
Throughout the paper, the expression ``non-crossing matching'', or just the word ``matching'',
will only refer to matchings of this kind,
and to their combinatorial and topological generalizations that will be defined below
(unless specified otherwise).
The \textit{size} of such a matching is $k$, the number of edges.
It is well-known that the number of matchings of $X_{2k}$ is
the $k$th Catalan number $C_k = \frac{1}{k+1}\binom{2k}{k}$~\cite[A000108]{oeis}.
Three examples of matchings of size~$8$ are shown in Figure~\ref{fig:first_examples_n}.
\begin{figure}[h]
$$\includegraphics[width=130mm]{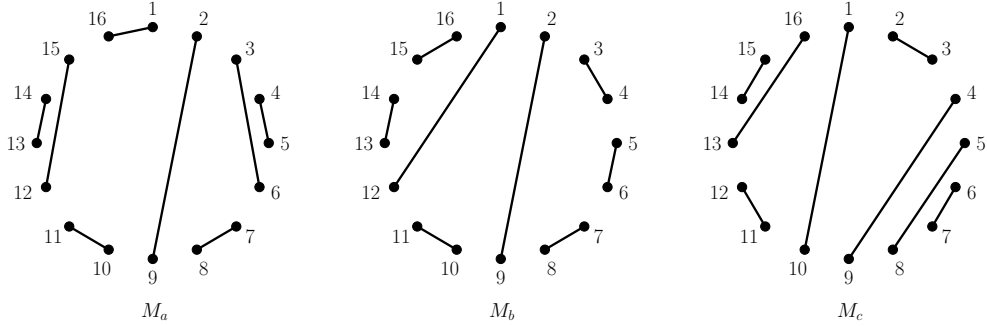}$$
\caption{Three examples of matchings of size $8$. $M_b$ and $M_c$ are disjoint compatible.}
\label{fig:first_examples_n}
\end{figure}

Two matchings $M$ and $M'$ of $X_{2k}$ are \textit{disjoint compatible}
if they do not have common edges (\textit{disjoint}),
and no edge of $M$ crosses an edge of $M'$ (\textit{compatible}).
In Figure~\ref{fig:first_examples_n},
the matchings $M_a$ and $M_b$ are not disjoint
($P_2 P_9$ is a common edge);
the matchings $M_a$ and $M_c$ are disjoint but not compatible
($P_3 P_6$ of $M_a$ and $P_4 P_9$ of $M_c$ cross each other);
the matchings $M_b$ and $M_c$ are disjoint compatible.

The \textit{disjoint compatibility graph} of matchings of size $k$ is the graph
whose vertices correspond to all such matchings of $X_{2k}$,
and two vertices are adjacent if and only if the corresponding matchings are disjoint compatible.
This graph will be denoted by $\dcm_k$.
The graph $\dcm_4$ is shown in Figure~\ref{fig:dcm4}.
It is clear that, while we consider point sets in convex position, the graph $\dcm_k$
does not depend on a specific set $X_{2k}$.
Occasionally we shall adopt the terminology from graph theory for the matchings
and say, for example,
``matching $M$ has degree $d$'',
``two matchings, $M$ and $N$ are connected''
to mean ``the vertex corresponding to $M$ in $\dcm_k$ has degree $d$'',
``the vertices corresponding to $M$ and $N$ in $\dcm_k$ are connected'', etc.
In particular,
``$M'$ is adjacent to $M$'' and ``$M'$ is a neighbor of $M$''
are synonyms of
``$M'$ is disjoint compatible to $M$''.

\begin{figure}[h]
$$\includegraphics[width=120mm]{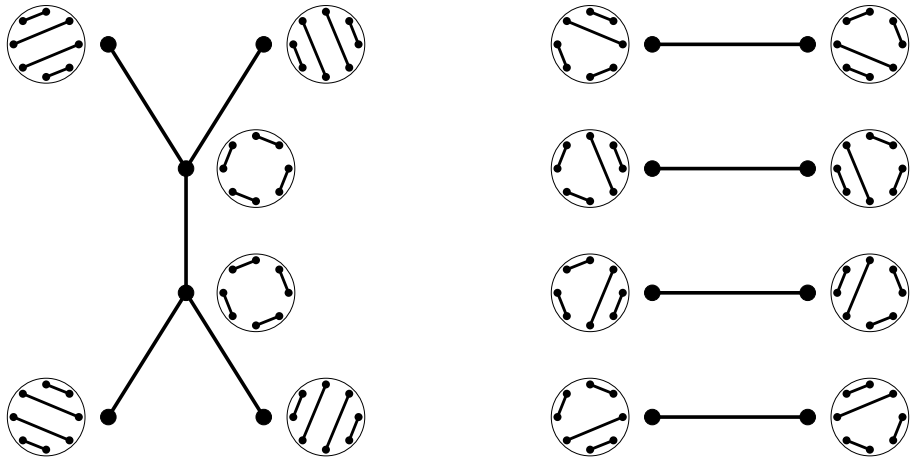}$$
\caption{The graph $\dcm_4$.}
\label{fig:dcm4}
\end{figure}

\medskip

In this paper we study the graphs $\dcm_k$,
mainly aiming for a description of their connected components
from the point of view of their structure, order
(that is, the number of vertices),
and isomorphism classes.
Our main results are the following theorems.

\begin{theorem}\label{thm:main_theorem}
For each $k \geq 9$, the connected components of $\dcm_k$
form exactly three isomorphism classes.
Specifically, there are several isomorphic components of the smallest order,
several isomorphic components of the medium order,
and one component of the biggest order.
\end{theorem}

In accordance to the orders, we call the components \emph{small}, \emph{medium} and \emph{big}.
The components of $\dcm_k$ follow different regularities for odd and for even values of $k$,
as specified in the next two theorems.
In fact, some of these regularities also hold for smaller values of $k$,
and thus we extend this notation for all values of $k$.
Namely, the components of the smallest order are called \emph{small};
the components of the next order are called \emph{medium};
all other components are called \textit{big}.
It was found by direct inspection and by a computer program that for $1 \leq k \leq 8$
the number of isomorphism classes of the components of $\dcm_k$ is 
as follows:

\medskip

\begin{center}
\begin{tabular}{c|cccccccc}
$k$   & $1$ & $2$ & $3$ & $4$ & $5$ & $6$ & $7$ & $8$    \\ \hline
Number of isomorphism classes \\ of the components of $\dcm_k$ & $1$ & $1$ & $2$ & $2$ & $3$ & $3$ & $4$ & $4$ \\ 
\end{tabular}
\end{center}

\medskip

However, as stated in Theorem~\ref{thm:main_theorem},
for all $k\geq 9$, $\dcm_k$ has components of exactly three kinds:
several small components, several medium components, and one big component.

\textit{Throughout the paper, we denote $ \ell = \left\lceil \frac{k}{2} \right\rceil$.}

\begin{theorem}\label{thm:main_odd}
Let $k$ be an odd number, $ \ell = \left\lceil \frac{k}{2} \right\rceil$.
\begin{enumerate}
  \item
  The small components of $\dcm_k$ are isolated vertices. \\
  The number of such components is $\frac{1}{\ell}\binom{4\ell-2}{\ell-1}$.
  \item
  For $k \geq 3$, the medium components of $\dcm_k$ are stars of order
  $\ell$ (that is, $K_{1, \ell-1}$). \\
  For $k \geq 5$, the number of such components is $(2\ell-1) \cdot 2^{\ell-3}$.
\end{enumerate}
\end{theorem}

\begin{theorem}\label{thm:main_even}
Let $k$ be an even number, $ \ell = \left\lceil \frac{k}{2} \right\rceil$.
\begin{enumerate}
  \item
  The small components of $\dcm_k$ are pairs (that is, components of order $2$). \\
  The number of such components is $\ell \cdot 2^{\ell-1}$.
  \item
  For $k \geq 4$, the medium components of $\dcm_k$ are of order $6\ell-6$.\footnote{
  The structure of the medium components for even $k$ will be described below, in Corollary~\ref{thm:middle_even_struct}.} \\
  For $k \geq 6$, the number of such components is $\ell \cdot 2^{\ell-2}$.
\end{enumerate}
\end{theorem}

The enumerational results from these theorems,
and exceptional values observed for small values of $k$,
are summarized in Tables~\ref{tab:odd} and~\ref{tab:even}.
As mentioned above, for $k=7$ and for $k=8$ two big components are of different order.

\renewcommand{\arraystretch}{1.5}
\begin{table}[h]
\centering
\begin{tabular}{|c||c|c|c|c|c|c|c||c|}
  \hline
$k$   & $1$ & $3$ & $5$ & $7$ & $9$ & $11$ & $\dots$ & General formula  \\ \hline
$\ell = \frac{k+1}{2}$  & $1$ & $2$ & $3$ & $4$ & $5$ & $6$ & $\dots$ &  \\ \hline \hline
Small components: order   & $1$ & $1$ & $1$ & $1$ & $1$ & $1$ & $\dots$ & $1$ \\ \hline
Small components: number  & $1$ & $3$ & $15$ & $91$ & $612$ & $4389$ & $\dots$ & $\frac{1}{\ell}\binom{4\ell-2}{\ell-1}$ \\  \hline
Medium components: order   & $-$ & $2$ & $3$ & $4$ & $5$ & $6$ & $\dots$ & $\ell$ \ (for $\ell\geq 2$) \\ \hline
Medium components: number & $-$ & $1$ & $5$ & $14$ & $36$ & $88$ & $\dots$ & $(2\ell-1) \cdot 2^{\ell-3}$ \ (for $\ell\geq 3$)\\ \hline
\end{tabular}
\caption{The summary of enumerational results for odd $k$ (Theorem~\ref{thm:main_odd}).}
\label{tab:odd}
\end{table}
\begin{table}[h]
\centering
\begin{tabular}{|c||c|c|c|c|c|c|c||c|}
  \hline
$k$   & $2$ & $4$ & $6$ & $8$ & $10$ & $12$ & $\dots$ & General formula  \\ \hline
$\ell = \frac{k}{2}$  & $1$ & $2$ & $3$ & $4$ & $5$ & $6$ & $\dots$ &  \\ \hline \hline
Small components: order   & $2$ & $2$ & $2$ & $2$ & $2$ & $2$ & $\dots$ & $2$ \\ \hline
Small components: number  & $1$ & $4$ & $12$ & $32$ & $80$ & $192$ & $\dots$ & $\ell\cdot 2^{\ell-1}$ \\  \hline
Medium components: order   & $-$ & $6$ & $12$ & $18$ & $24$ & $30$ & $\dots$ & $6 \ell - 6$ \ (for $\ell\geq 2$) \\ \hline
Medium components: number & $-$ & $1$ & $6$ & $16$ & $40$ & $96$ & $\dots$ & $\ell\cdot 2^{\ell-2}$ \ (for $\ell\geq 3$) \\  \hline
\end{tabular}
\caption{The summary of enumerational results for even $k$ (Theorem~\ref{thm:main_even}).}
\label{tab:even}
\end{table}
\renewcommand{\arraystretch}{1}

As stated in Theorem~\ref{thm:main_theorem}, for $k\geq 9$ there is only one big component.
Thus, its order is the number of vertices that do not belong to small and medium components.
In 
Proposition~\ref{thm:big}
we will show that the order of the
big component is indeed larger than that of medium or small components.

\subsection{Background and motivation}\label{sec:background}

The general notion of disjoint compatibility graphs was defined by
Aichholzer et al.~\cite{aichholzer} for sets of $2k$ points in general
(not necessarily convex) position. While they showed that for odd $k$
there exist isolated matchings, they posed the
\emph{Disjoint Compatible Matching Conjecture} for even $k$:
For every non-crossing matching of even size,
there exists a disjoint compatible non-crossing matching.
This conjecture was recently answered in
the positive by Ishaque~et~al.~\cite{ishaque}. In that paper it was
stated that for even $k$ ``it remains an open problem whether [the
disjoint compatibility graph] is always connected.''  It follows from
our results that for sets of $2k$ points in convex position,
$\dcm_{k}$ is \textbf{always} disconnected, with the exception of $k=1$ and $2$.

Both concepts, disjointness and compatibility, can be found in
generalized form for various geometric structures. For example, two
triangulations are compatible if one can be obtained from the other by
removing an edge in a convex quadrilateral and replacing it by the
other diagonal. This operation is called a flip and it is well known
that in that way any triangulation of the given set of $n$ points can
be obtained from any other triangulation of the same set
with at most $O(n^2)$ flips,
see e.~g.~\cite{Hurtado1999}.
Similar results exist, for example, for spanning trees~\cite{aah2002}
and between matchings and other geometric graphs~\cite{aght, houle}.

It is convenient to describe such results in terms of
\textit{reconfiguration graphs},
whose vertices correspond to all configurations under discussion,
two vertices being adjacent when the corresponding configurations
can be obtained from each other by certain operation (``reconfiguration'').
In these terms, the above mentioned result about
flips in triangulations can be stated as follows:
the flip graph of triangulations is connected with diameter $O(n^2)$.

Some kinds of reconfiguration graphs of non-crossing matchings were studied as well.
Hernando~et~al.~\cite{hernando} studied graphs of 
non-crossing perfect matchings of $2k$ points in \textbf{convex} position 
with respect to reconfiguration of the kind
$M' = M - (a,b) - (c, d) + (b, c) + (d, a)$.
In particular, they proved that such a graph is $(k-1)$-connected and has diameter $k-1$,
and it is bipartite for every $k$.
Aichholzer~et~al.~\cite{aichholzer} considered graphs of 
non-crossing perfect matchings of $2k$ points in \textbf{general} position,
where the matchings are adjacent if and only if they are compatible (but not necessarily disjoint).
They showed that in such a graph
there always exists a path of length at most $O(\log k)$ between any two matchings. 
Hence, such graphs are connected with diameter~$O(\log k)$; 
lower bound examples with diameter
$\Omega(\log k/ \log \log k)$ were found by Razen~\cite[Section~4]{Razen2008}.

In general, the number of non-crossing matchings of a point set depends on its order type.
In contrast to the case of point sets in convex position,
for general point sets no exact bounds are known.
Sharir and Welzl~\cite{Sharir2006} proved that any set of $n$ points
has $O(10.05^n)$ non-crossing matchings.
Garc{\'i}a et~al.~\cite{garcia} showed that the number of non-crossing matchings is minimal
when the points are in convex position
(then, as mentioned above, the number of matchings is $C_{n/2} = \Theta^*( 2^n )$),
and constructed a family of examples with $\Theta^*(3^n)$ matchings.
In these papers, bounds for similar problems concerning
other geometric non-crossing structures
(triangulations, spanning trees, etc.) are also found.

A generalization for matchings are \textit{bichromatic matchings}. There the
point set consists of $k$ red and $k$ blue points, and an edge always
connects a red point to a blue point. It has recently been shown by
Aloupis et al.~\cite{aloupis} that the graph of compatible (but not
necessarily disjoint) bichromatic matchings is connected. Moreover,
the diameter of this graph is $O(k)$, see~\cite{abhpv2013}.
On the other hand, certain bichromatic point sets have only one bichromatic matching:
such sets were characterized in~\cite{br}.

\smallskip

From the combinatorial point of view,
non-crossing matchings of points in convex position are
identical to so called \emph{pattern links}.
Pattern links of size $k$ form a basis for Temperley-Lieb algebra
$\mathrm{TL}_k(\delta)$ that was first defined in~\cite{temperley}, and has
numerous applications in mathematical physics, knot theory, etc.
Pattern links also have a close relation with alternating sign matrices (ASMs),
fully packed loops (FPLs), and other combinatorial structures. For more
information see the survey article by Propp~\cite{propp}.
Di~Francesco et~al.~\cite{francesco} constructed a bijection between FPLs
with a link pattern consisting of three nested sets of sizes $a$, $b$
and $c$ and the plane partitions in a box of size $a \times b \times c$.
Wieland~\cite{wieland} proved that the distribution of link
patterns corresponding to FPLs is invariant under dihedral relabeling.
A connection between the distribution of link patterns of FPLs
and
ground-state vector of $O(1)$ loop model
from statistical mechanics
was intensively studied in the last years:
see, for example, a proof of
Razumov-Stroganov conjecture~\cite{razumov}
(which can be also expressed in terms of reconfiguration)
by
Cantini and Sportiello~\cite{cantini}.

\medskip

Thus, our contribution is twofold.
First, from the combinatorial point of view, we have structural results
that provide a new insight into combinatorics of non-crossing partitions.
Second, our work is a contribution to the study of straight-line graph drawings.
While it applies only to matchings of points in convex position, 
certain observations may be carried over or generalized for general 
sets of points, and, thus, they could be possibly useful for the study
of disjoint compatibility of geometric matchings in general.

\subsection{Outline of the paper.}\label{sec:plan}
The paper is organized as follows.
In Section~\ref{sec:definitions}
we introduce notion necessary for the proofs of the main theorems,
and prove some preliminary results.
One important notion there will be that of \textit{block}:
two edges that connect four consecutive points of $X_{2k}$, the first with the fourth, and the second with the third.
In particular, it will be observed that if a matching $M$ has a block,
then in any matching disjoint compatible to $M$ the points of the block can be reconnected in a unique way.
Thus, presence of blocks puts restrictions on potential matchings disjoint compatible to $M$.

In Section~\ref{sec:small} we describe certain kinds of matchings
and show that they belong to components of the smallest possible order
($1$ or $2$, depending on the parity of $k$).
In Section~\ref{sec:middle}, we describe other kinds of matchings,
and prove that, for fixed $k$, all the connected components that contain such matchings are isomorphic.
Enumerational results from these sections fit the rows of
Tables~\ref{tab:odd} and~\ref{tab:even} that correspond to medium components.
Finally, in Section~\ref{sec:big}, we prove that
for $k \geq 9$
all the matchings that do not belong to either of the kinds
from Sections~\ref{sec:small} and~\ref{sec:middle}, form one connected component of big order
(essentially, we prove that all such matchings are connected by a path
to so called \textit{rings}).
In particular, this implies that no other orders exist,
and that all the small and medium components are, indeed, described in Sections~\ref{sec:small} and~\ref{sec:middle}.
Thus, this accomplishes the proof of Theorems~\ref{thm:main_theorem},~\ref{thm:main_odd} and~\ref{thm:main_even}.
In the concluding Section~\ref{sec:conc}, we 
showing more enumerational results related to $\dcm$,
briefly discuss the case of ``almost perfect'' matchings of sets that have odd number of points,
and suggest several problems for future research.

\section{Further definitions and basic results}\label{sec:definitions}

\subsection{Flipping}\label{sec:flip}

If an edge of a matching connects two consecutive points of $X_{2k}$,
it is a \textit{boundary edge}, otherwise it is a \textit{diagonal edge}. 
(We regard $X_{2k}$ as a cyclic structure.
Thus, the points $P_{2k}$ and $P_1$ are also considered consecutive.
Moreover, the arithmetic of the labels will be modulo $2k$.
Yet we write $P_{2k}$ rather than $P_{0}$.)
In the matching $M_a$ in Figure~\ref{fig:chains_1},
the edges $P_3 P_8$ and $P_{13} P_{16}$ are diagonal edges,
and all other edges are boundary edges.
A pair of consecutive points not connected by an edge is a \textit{skip}.
For each $k\geq 2$ there are two matchings with only boundary edges,
which we call \textit{rings}.
Notice that the two rings are disjoint compatible to each other.

The definition of disjoint compatible matchings can be rephrased as follows.
\begin{observation}\label{thm:def_cycles}
Let $M$ and $M'$ be matchings of $X_{2k}$.
$M$ and $M'$ are disjoint compatible
if and only if
$M \cup M'$
is a union of pairwise disjoint cycles
that consist alternatingly of edges of $M$ and $M'$.
\end{observation}

See Figure~\ref{fig:chains_1} for an example.

\begin{figure}[h]
$$\includegraphics[width=130mm]{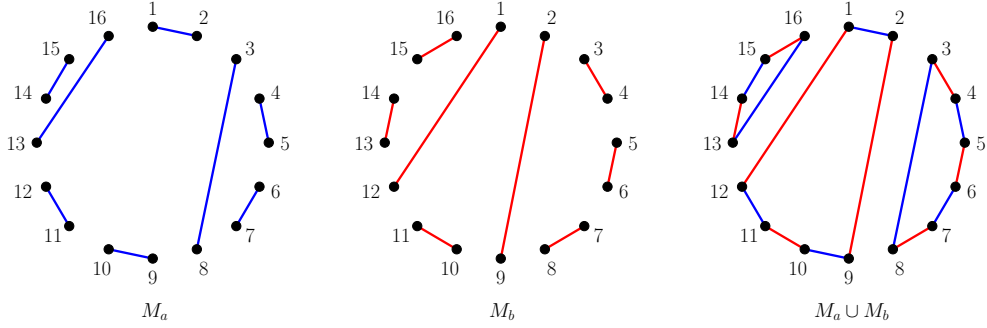}$$
\caption{The union of disjoint compatible matchings is a union of disjoint alternating cycles.}
\label{fig:chains_1}
\end{figure}

Let $M$ be a matching of $X_{2k}$, and let $Y$ be a subset of $X_{2k}$ of size $2m$ ($2 \leq m \leq k $)
whose members are labeled cyclically by $Q_1, Q_2, \dots, Q_{2m}$. 
(In other words, $Q_a=P_{i_a}$, and $\{i_1, i_2, \dots, i_{2m}\}$ is a subset of $\{1, 2, \dots, 2k\}$
with the induced cyclic order.)
If
$N=\{Q_1Q_2,$ $Q_3Q_4,$ $Q_5Q_6, \dots,$ $Q_{2m-3}Q_{2m-2},$ $Q_{2m-1}Q_{2m}\}$
is a subset of $M$,
and the convex hull of $Y$ does not intersect any other edge of $M$,
we say that $N$ is a \emph{flippable set}.
Replacing the set $N$ by the set
$N'=\{Q_2Q_3,$ $Q_4Q_5,$ $Q_6Q_7, \dots,$ $Q_{2m-2}Q_{2m-1},$ $Q_{2m}Q_{1}\}$
is a \emph{flip} of $N$.

\begin{proposition}\label{thm:flip}
Let $M$ and $M'$ be non-crossing matchings of $X_{2k}$.
$M$ and $M'$ are disjoint compatible
if and only if
there is a (uniquely determined) partition of $M$ into flippable sets with pairwise disjoint convex hulls
so that $M'$ is obtained from $M$ by flipping them.
\end{proposition}

\begin{proof} $[\Leftarrow]$ In such a case, $M\cup M'$ is a union of pairwise disjoint cycles
as in Observation~\ref{thm:def_cycles}.
$[\Rightarrow]$ Taking the edges of $M$ that belong to a cycle as in Observation~\ref{thm:def_cycles},
we obtain a flippable set.
Since these cycles are connected components of $M \cup M'$,
the partition of $M$ into flippable sets is uniquely determined by $M$ and $M'$.
Since the cycles are disjoint, these flippable sets have disjoint convex hulls.
\end{proof}

A partition as in Proposition~\ref{thm:flip} will be called a \textit{flippable partition}.
Notice that a flippable set can not always be extended to a flippable partition.
For example, the set $T=\{P_{1}P_{2}, P_{3}P_{8}, P_{13}P_{16}\}$
from the matching $M_a$
in Figure~\ref{fig:chains_1}
is a flippable set, but there is no flippable partition that contains this set
because there is no flippable set that contains $\{P_{14}P_{15}\}$
and doesn't cross $T$.

\subsection{Merging and splitting of matchings}\label{sec:merge}

In some cases we need to split a matching into two submatchings,
or to merge two matchings into one matching.
Let $L$ and $N$ be non-empty disjoint subsets (submatchings) of a matching $M$
so that their union is $M$,
and so that $L$ can be separated from $N$ by a line.
In such a case we write $M=L+N$, or $N=M-L$,
and say that $L+N$ is a \textit{decomposition} of $M$.
If we want to treat $L$ and $N$ as matchings of respective sets of points,
we need to indicate how the labeling of $M$ is split into,
or merged from the respective labelings of $L$ and $N$.
We formalize the merging of two matchings in the following way.
Let $L$ be a matching of $2r$ points $\{R_1, R_2, \dots, R_{2r}\}$,
and let $N$ be a matching of $2s$ points $\{S_1, S_2, \dots, S_{2s}\}$.
A matching $M$ obtained by \emph{insertion of $N$ into $L$ between the points $R_a$ and $R_{a+1}$}
is the matching of $2k=2r+2s$ points $P_1, P_2, \dots, P_{2k}$
obtained by relabeling (and putting in convex position) from
$R_1, R_2 \dots, R_{a}, S_1, S_2, \dots, S_{2s}, R_{a+1}, R_{a+2}, \dots, R_{2r}$ (in this order),
such that $P_i P_j$ is an edge if and only if the corresponding points are connected in $L$ or in $N$.
If $N$ is inserted into $L$ between $R_{2r}$ and $R_{1}$,
we have $2s+1$ possibilities to choose the point corresponding to $P_1$:
$R_1$ or either of the points $S_i$.
A similar procedure can be described for splitting a matching (we omit the details).

In some cases specifying the labeling upon merging or splitting will not be essential.
For example, in some proofs we split a matching $M$ into two submatchings $L$ and $N$,
modify both parts, and then merge them again.
In such a case we only need to make sure that when the parts are merged,
their vertices are labeled in the same way as before the splitting.
Assuming this convention, we mention the following obvious fact.
\begin{observation}\label{thm:split}
Let $M$ be a matching, and suppose that $L+N$ is its decomposition.
If $L'$ is a matching disjoint compatible to $L$,
and $N'$ is a matching disjoint compatible to $N$,
then $L'+N'$ is disjoint compatible to $M$.
\end{observation}

If we start with a matching $M_0$, and perform insertion several times
(each time the inserted matching, the place of insertion, and, if needed, the labeling are specified),
obtaining thus a sequence of matchings $M_1, M_2, \dots$,
then for each edge $e$ of $M_0$, each of the members of this sequence has an
edge corresponding to $e$ in the obvious sense.

\subsection{Combinatorial and topological matchings}\label{sec:gamma}
For the sets of points in convex position,
the notions of non-crossing matchings
and that of disjoint compatible matchings
are in fact purely combinatorial,
since being two edges crossing or non-crossing is completely determined by the labels of their endpoints.
Indeed, let $X_{2k}$ be just the set $\{1, 2, \dots, 2k\}$.
Two disjoint pairs of members of $X_{2x}$, $\{a_1,a_2\}$ and $\{b_1,b_2\}$, are \textit{crossing} if,
when ordered with respect to the usual cyclic order of $X_{2k}$,
they form a sequence of the form $abab$.
A \textit{combinatorial non-crossing matching} of $X_{2k}$ is its partition $M$ into $k$ disjoint non-crossing pairs.
Two such matchings, $M$ and $M'$, are disjoint compatible
if no pair belongs to them both,
and no pair from $M$ crosses a pair from $M'$.

Combinatorial non-crossing matchings can be represented not only by straight-line (``geometric'') drawings,
but also by more general ``topological drawings'', as follows.
Let $\mathbf{\Gamma}$ be a closed Jordan curve,
and let $X_{2k}=\{P_1, \dots, P_{2k}\}$
be a set of points that lie (say, clockwise) on $\mathbf{\Gamma}$ in this cyclic order.
Denote by $\mathbf{O(\Gamma)}$ the interior, that is, the region bounded by $\mathbf{\Gamma}$.
A \textit{topological non-crossing matching} is a set of $k$
non-intersecting Jordan curves that connect pairs of these points,
and whose interior lies in $\mathbf{O(\Gamma)}$.
Since $\mathbf{O(\Gamma)}$ is homeomorphic to an open disc (by the Jordan-Schoenflies theorem),
each topological non-crossing matching can be continuously transformed into a geometric non-crossing matching.
Notice, however, that (in contrast to geometric matchings)
two topological matchings (on the same $X_{2k}$ and $\mathbf{\Gamma}$)
that correspond to disjoint compatible combinatorial matchings
might have crossing arcs.

In what follows, by a (non-crossing) matching we usually mean either a combinatorial non-crossing matching as described above,
or any of its topological or straight-line representations.
When a specific kind of drawing should be considered, we will mention it explicitly.

\subsection{The map and the dual tree}\label{sec:dual}
Consider a topological non-crossing matching $M$ of size $k$.
Then the union of $\mathbf{\Gamma}$ and the members of $M$
form a planar map in $\mathbf{O(\Gamma)}$.
This map has $k+1$ faces.
The boundary of each face consists of one or several pieces of $\mathbf{\Gamma}$
and one or several edges of $M$.
Each edge belongs to exactly two faces.
A face that has more than one edge will be called an \textit{inner face};
a face that has exactly one edge (which is then necessarily a boundary edge) will be called a \textit{boundary face}.
Notice that any flippable set is a subset of the set of edges that belong to one (inner) face.

Consider the dual graph of this map, regarded as a combinatorial embedding
(that is, for each vertex $v$ the cyclic order $\phi(v)$ of edges incident to $v$ is specified)
with labeled \textit{edge sides}.
This graph $T$ is a tree: it is easy to see that $T$ is connected and
acyclic, as removal of any edge of $T$ disconnects it.
It will be called the \emph{dual tree} of $M$, and denoted by $D(M)$.
Since each edge of $D(M)$ crosses exactly one edge of $M$,
the points of $X_{2k}$ correspond to the edge sides of $D(M)$ in a natural way;
therefore, we use the indices of the points as labels of the edge sides.
The boundary edges of $M$ correspond
to the edges of $D(M)$ incident to leaves,
and, thus, there is also a clear correspondence of the boundary edges of $M$
to the leaves of $D(M)$.
The skips of $M$ correspond to the \textit{wedges} --
pairs of edges incident to a vertex $v$, consecutive in $\phi(v)$
(geometrically,
in case of straight-line drawing,
the wedges are angles formed by edges incident to the same vertex $v$,
with the center in $v$).
In Figure~\ref{fig:first_examples_graph_6}(a, b),
a matching $M$ (black) and its dual tree $D(M)$ (blue) are shown.
\begin{figure}[h]
$$\includegraphics[width=140mm]{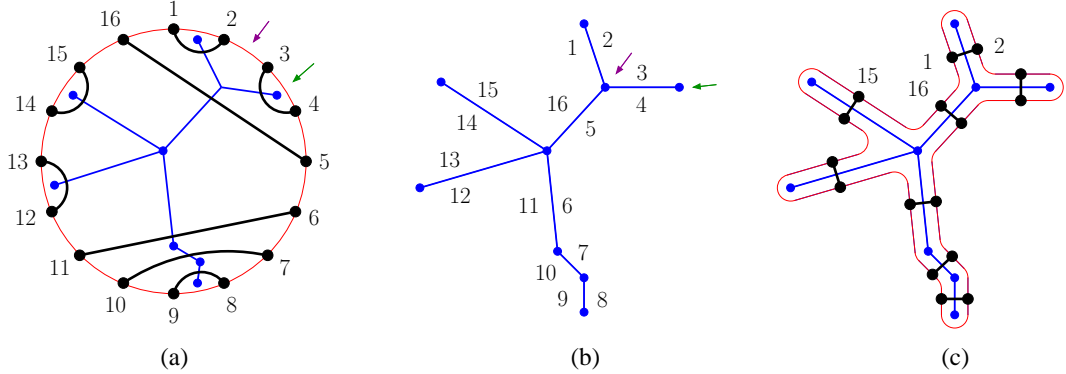}$$
\caption{(a) A matching.
(b) Its dual tree.
(c) Reconstructing the matching from its dual tree.}
\label{fig:first_examples_graph_6}
\end{figure}

Combinatorial embeddings of trees with $k+1$ vertices and one marked edge side
are in bijection with matchings of size $k$.
Notice that one marked edge side
(we use the label $1$ as the mark)
in such an embedding $T$ determines
a labeling of edge sides of $T$
by $\{1, 2, \dots, 2k\}$
that agrees with a cyclic ordering of edge sides
determined by a \textit{clockwise double edge traversal}.\footnote{In a double edge traversal,
each edge is visited twice: once for each direction.
After visiting an edge $e=v_1v_2$ from $v_1$ to $v_2$,
we visit the edge $v_2v_3$,
the successor of $e$ in $\phi(v_2)$,
from $v_2$ to $v_3$.
}
Figure~\ref{fig:first_examples_graph_6}(c) shows how, given such a combinatorial embedding of a tree $T$,
one can construct the matching $M$ such that $D(M)=T$.
First, we take a drawing of $T$
(for example, a straight-line drawing -- it is well-known that such a drawing always exists)
and slightly inflate its edges.
The boundary of the obtained shape is a closed Jordan curve $\mathbf{\Gamma}$,
it can be seen as a route of the double edge traversal.
For each edge of $T$, we put a point on $\mathbf{\Gamma}$ on each of its sides,
and connect such pairs by arcs.
As explained above, the edge sides of $T$ are labeled by $\{1, 2, \dots, 2k\}$.
The point that lies on the edge side $i$ will be labeled by $P_i$.
The set of arcs is now a non-crossing matching whose dual tree is $T$.
This topological matching can be converted now into
a straight-line matching of points in convex position as explained above.
Without a marked edge side, a combinatorial embedding determines a class of
\emph{rotationally equivalent} matchings,
that is,
matchings that can be obtained from each other by a cyclic relabeling of vertices.
We summarize our observations as follows.

\begin{observation}\label{thm:dual}
\mbox{}
\begin{enumerate}
\item The
correspondence $M \mapsto D(M)$
 is a bijection between
 combinatorial embeddings of trees with $k+1$ vertices and one marked edge side
and non-crossing matchings of size $k$.
\item Two non-crossing matchings, $M_1$ and $M_2$, have the same
non-labeled
dual tree if and only if they are rotationally equivalent.
\end{enumerate}
\end{observation}

\subsection{Blocks and antiblocks}\label{sec:basic}

\noindent\textbf{Definition.} Let $M$ be a matching of $X_{2k}$, $k\geq 2$.
\begin{enumerate}
  \item A \emph{block} is a pair of edges of $M$ of the form $\{P_i P_{i+3}, P_{i+1} P_{i+2}\}$.
  \item An \emph{antiblock} is a pair of edges of $M$  of the form $\{P_i P_{i+1}, P_{i+2} P_{i+3}\}$.
  \item A \emph{separated pair} is a block or an antiblock.
\end{enumerate}

For example, in the matching $M_a$ from Figure~\ref{fig:chains_1},
$\{P_{13}P_{16}, P_{14}P_{15}\}$ is a block,
and $\{P_{4}P_{5}, P_{6}P_{7}\}$ is an antiblock.
If we have a separated pair on points
$P_i, P_{i+1}, P_{i+2}, P_{i+3}$,
then they will be called, respectively,
the first, the second, the third, and the fourth points of the separated pair.
For a block $K=\{P_i P_{i+3}, P_{i+1} P_{i+2}\}$,
the edge $P_i P_{i+3}$ is the \textit{outer}, and the edge $P_{i+1} P_{i+2}$ is the \textit{inner} edge of $K$.\footnote{
A special case is $k=2$.
Consider $M=\{P_1P_2, P_3P_4\}$.
The whole matching is both a block and an antiblock.
For $M$ as a block, $P_2$ or $P_4$ can be taken as the first point.
For $M$ as an antiblock, $P_1$ or $P_3$ can be taken as the first point.
The case of $M=\{P_1P_4, P_2P_3\}$ is similar.
}
For $k>3$ two blocks in a matching are necessarily disjoint, while two antiblocks can share an edge.
The block $\{P_i P_{i+3}, P_{i+1} P_{i+2}\}$ and the antiblock $\{P_i P_{i+1}, P_{i+2} P_{i+3}\}$
are \textit{flips} of each other.
The special role of blocks is due to the following observation.

\begin{observation}\label{thm:obs_block_antiblock}
Let $M$ and $M'$ be two disjoint compatible matchings.
If $M$ has a block \\ $\{P_i P_{i+3}, P_{i+1} P_{i+2}\}$,
then $M'$ has an antiblock $\{P_i P_{i+1}, P_{i+2} P_{i+3}\}$.
\end{observation}
\begin{proof}
Consider a flippable partition of $M$.
The only flippable set of $M$ that contains the edge $P_{i+1} P_{i+2}$
is the block $\{P_i P_{i+3}, P_{i+1} P_{i+2}\}$.
Upon flipping, an antiblock on these points is obtained.
\end{proof}

Given a matching $M$ of size $k$,
we can obtain a matching of size $k+2$ by inserting a matching $K$ of size $2$.
When essential, we can use the rule of relabeling vertices as explained in Section~\ref{sec:merge}.
However, instead of specifying a labeling of $K$,
we say that we insert a block or an antiblock into $M$
in accordance to the shape formed by the edges corresponding to $K$ in $M+K$.

%
The definition of the dual tree and the
correspondence between elements of $M$ and $D(M)$
(explained before Observation~\ref{thm:dual})
allow to identify elements of $D(M)$ that correspond
to separated pairs.

\medskip

\noindent\textbf{Definition.}
Let $T$ be a combinatorial embedding of a tree.
\begin{enumerate}
  \item A \emph{$k$-branch} in $T$ is
a path $v_1v_2 \dots v_{k+1}$ of length $k$ whose one end ($v_{k+1}$) is a leaf in $T$,
and all the inner vertices ($v_2, v_3, \dots, v_k$)
have degree $2$.
A $k$-branch will be given by the list of its vertices, starting from $v_1$.
  \item A \emph{V-shape} in $T$ is a path $v_1v_2v_3$
  such that $v_1$ and $v_3$ are leaves in $T$,
  and the edge $v_2v_3$ follows the edge $v_2v_1$ in $\phi(v_2)$
  (in other words, $v_1v_2v_3$ is a wedge).
 A V-shape will be given by the list of its vertices in this order,
 corresponding to the clockwise double edge traversal: $v_1v_2v_3$.
\end{enumerate}

\begin{observation}\label{thm:block_antiblock_dual_elements}
Blocks in $M$ correspond to $2$-branches  in $D(M)$.
Antiblocks in $M$ correspond to V-shapes in $D(M)$.
\end{observation}

Suppose now that $T$ is a combinatorial embedding of a tree,
and we want to add a $k$-branch or a V-shape to $T$.
The following convention will be adopted.
We say that an embedding $T'$ is obtained from $T$ by attaching
a $k$-branch $v_1v_2 \dots v_{k+1}$ to vertex $w$ of $T$ in the wedge $w_1ww_2$,
if
(1) $v_1=w$,
(2) the vertices $v_2, \dots, v_{k+1}$ are vertices of $T'$ but not of $T$,
and
(3) for $w$ in $T'$ we have $ww_1 \prec wv_2 \prec ww_2$ in $\phi(w)$.
We say that an embedding $T'$ is obtained from $T$ by attaching
a V-shape $v_1v_2v_3$ to vertex $w$ of $T$ in the wedge $w_1ww_2$,
if
(1) $v_2=w$,
(2) the vertices $v_1, v_{3}$ are vertices of $T'$ but not of $T$,
and
(3)
for $w$ in $T'$ we have $ww_1 \prec wv_1 \prec wv_3 \prec ww_2$ in $\phi(w)$.

\begin{observation}\label{thm:block_antiblock_dual_insertion}
Let $M$ be a matching.

Inserting a block (respectively, an antiblock) in $M$
between the points $P_{i}, P_{i+1}$ connected by an edge in $M$
corresponds to attaching
a $2$-branch (respectively, a V-shape)
to the leaf corresponding to this edge in $D(M)$.

Inserting a block (respectively, an antiblock) in $M$
between the points $P_{i}, P_{i+1}$ not connected in $M$
corresponds to attaching
a $2$-branch (respectively, a V-shape)
to the vertex
in the wedge corresponding to the skip
between $P_{i}$ and $P_{i+1}$ in $D(M)$.
\end{observation}

See Figure~\ref{fig:inserting_dual_1}:
$M$ is a matching of size $4$;
$M_a$ and $M_b$ are obtained from $M$ by inserting a block and, respectively, an antiblock
between $P_2$ and $P_3$ (not connected in $M$);
$M_c$ and $M_d$ are obtained from $M$ by inserting a block and, respectively, an antiblock
between $P_3$ and $P_4$ (connected in $M$).
\begin{figure}[h]
$$\includegraphics[width=165mm]{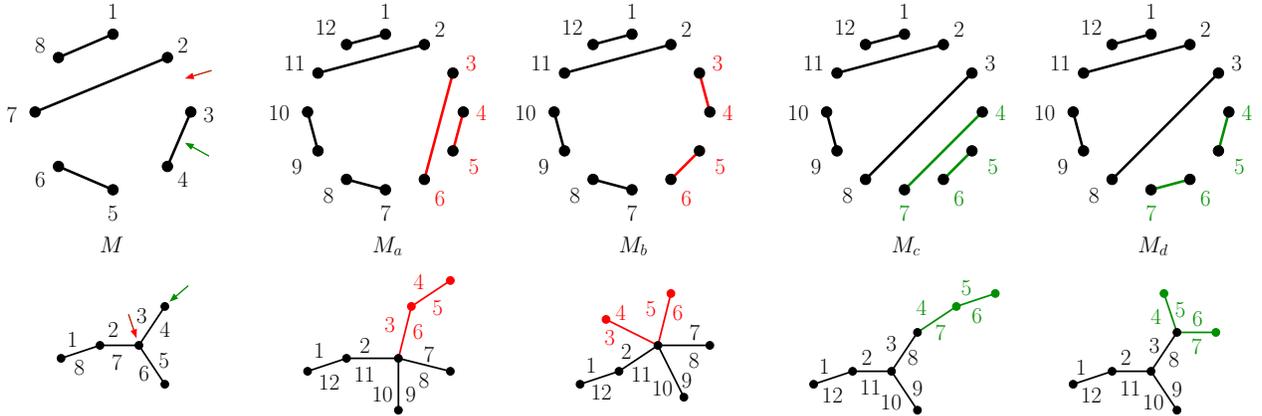}$$
\caption{Illustration to Observation~\ref{thm:block_antiblock_dual_insertion}.}
\label{fig:inserting_dual_1}
\end{figure}

\begin{proposition}\label{thm:exist_separated}
Let $M$ be a matching of size $k \geq 4$.
Then $M$ has at least two disjoint separated pairs.
\end{proposition}

\begin{proof}
If $M$ is a ring, the statement is clear.
Otherwise, $D(M)$ is not a star, and, thus, its diameter is at least $3$.
Let $v_1$ and $v_2$ be the leaves with the maximum distance in $D(M)$,
and let $u_1$ and $u_2$ be the vertices adjacent to them (respectively).
If $d(u_1)=2$, we have a $2$-branch in $D(M)$, and, therefore, a block in $M$.
If $d(u_1)>2$, we have a V-shape in $D(M)$,
and, therefore, an antiblock in $M$.
The same holds for $u_2$.
Since $u_1 \neq u_2$, these separated pairs are disjoint,
unless the whole $D(M)$ is the path $v_1u_1u_2v_2$.
But this situation is impossible since $k \geq 4$.
\end{proof}

\begin{proposition}\label{thm:block_keeps_degree}
Let $M$ be a matching of size $k$, and let $N=M+K$ where $K$ is a block.\footnote{
Since the place where $K$ was inserted is not specified, this means:
$N$ is some matching that can be obtained from $M$ by adding a block.
}
Then the degree of $N$ in $\dcm_{k+2}$ is equal to the degree of $M$ in $\dcm_{k}$.
\end{proposition}

\begin{proof}
The mapping $M' \mapsto M'+K'$, where $M'$ is a matching disjoint compatible to $M$,
and $K'$ is the antiblock that uses the same points as $K$,
is a bijection between
matchings disjoint compatible to $M$
and
matchings disjoint compatible to $N$.
\end{proof}

\begin{proposition}\label{thm:block_anti_connected}
Let $M$ be a matching of size $k$, and let $N=M+K$ where $K$ is a block or an antiblock.
If $M$ is connected (by a path) in $\dcm_{k}$ to $p$ matchings,
then $N$ is connected (by a path) in $\dcm_{k+2}$ to at least $p$ matchings.
\end{proposition}

\begin{proof}
Consider the mapping $M' \mapsto M'+K'$, where $M'$ is a matching connected by a path to $M$,
$K'=K$ if $d(M, M')$ is even,
and $K'$ is the flip of $K$ if $d(M, M')$ is odd.
It follows by induction on the distance
and by Observation~\ref{thm:split}
that for each $M'$,
the matching $M'+K'$ is connected by a path to $N$.
It is also clear that
this mapping is an injection.
\end{proof}

\section{Small components and vertices of small degree}\label{sec:small}

\subsection{General discussion}\label{sec:small_general}

A matching $M$ is \textit{isolated} if it is not disjoint compatible to any other matching of the same point set
(in other words, it corresponds to an isolated vertex of $\dcm_k$).
First we show that no isolated matchings of even size exists.\footnote{
As mentioned in the introduction, this claim also holds for matchings of points in general
(not necessarily convex) position~\cite[Theorem 1]{ishaque}.
However, since for the convex case the proof is very simple, we
present it here for completeness.}

\begin{proposition}\label{thm:even_never_isolated}
If $M$ is a matching of even size $k$, then there is at least one matching disjoint compatible to $M$.
\end{proposition}
\begin{proof} For $k=2$, the statement is obvious.
For $k \geq 4$: by Proposition~\ref{thm:exist_separated}, $M$ has a separated pair $K$.
Let $L = M - K$.
By induction, there exists a matching $L'$ disjoint compatible to $L$.
Now, $L'+K'$, where $K'$ is the flip of $K$,
is disjoint compatible to $M$ by Observation~\ref{thm:split}.
\end{proof}

In Section~\ref{sec:small_odd} we shall prove that for any odd $k$ there are isolated matchings of size $k$,
and in Section~\ref{sec:small_even} we shall prove that for any even $k$, $\dcm_k$ has connected components of size $2$.

First we derive certain situations in which a matching necessarily has at least one, or two,
disjoint compatible matchings.

\begin{proposition}\label{thm:few_blocks}
Let $M$ be a matching of size $k\geq 2$.
\begin{enumerate}
\item If $M$ has no blocks, then there are at least two matchings disjoint compatible with $M$.
\item If $M$ has exactly one block, then there is at least one matching disjoint compatible with $M$.
\end{enumerate}
\end{proposition}

\begin{proof}
For $k=2, 3$, we verify this directly (for $k=2$ the statement holds in a trivial way).
For $k \geq 4$, we prove the statement by induction
(notice that the induction applies not to $1.$ and $2.$ separately,
but rather to the whole statement).

\begin{enumerate}
\item Suppose that $M$ has no blocks.
If $M$ is a ring, then the claim is clear.
So, we assume that there is a diagonal edge $e = P_iP_j$.
Let $M_1$ and $M_2$ be the submatchings of $M$
on point sets $Y_1=\{P_{i+1}, P_{i+2}, \dots, P_{j-1}\}$
and $Y_2=\{P_{j+1}, P_{j+2}, \dots, P_{i-1}\}$ (respectively).
Since $M$ has no blocks, both these submatchings are of size at least $2$.


Consider the submatching $M_1$. If it has a block $K$, then its first point can be only one of the points
$P_{j-3}, P_{j-2},$ and $P_{j-1}$, because otherwise $K$ would be also a block of $M$.
It follows that $M_1$ has at most one block.
Therefore, it is not isolated by induction.
Similarly, $\{e\}\cup M_2$ has at most one block (its first point can be only $P_{i-1}$),
and therefore, it is also not isolated.
Denote by $M'_1$ a matching disjoint compatible to $M_1$, and by $M''_2$ a matching disjoint compatible to $\{e\}\cup M_2$.
Then $M'_1 +  M''_2$ is disjoint compatible to $M$.

Similarly, the submatchings $M_1 \cup \{e\}$ and $ M_2$ are non-isolated,
and $M''_1 +  M'_2$,
the merge of their respective disjoint compatible matchings, is disjoint compatible to $M$.

Thus we obtained two matchings, disjoint compatible to $M$.
They are indeed distinct because in $M'_1 +  M''_2$
the endpoints of $e$ are connected to points from $Y_2$,
and in $M''_1 +  M'_2$ to points of $Y_1$.

\item Suppose that $M$ has exactly one block $K$.
Let $L=M-K$.
Similarly to the reasoning from the previous paragraph, $L$ has at most one block, and, thus,
it is not isolated by induction.
Therefore, $M$ is also not isolated by Observation~\ref{thm:split}.
\end{enumerate}
\end{proof}

\textit{Remark.} The statements of Proposition~\ref{thm:few_blocks} cannot be strengthened
as the examples in Figure~\ref{fig:3ex} (for both even and odd $k$) show.
The matching $M_a$ has no blocks, and it has exactly two disjoint compatible matchings.
The matching $M_b$ has exactly one block, and it has exactly one disjoint compatible matching.
In order to see that, notice that a disjoint compatible matching for $M_a$ or for $M_b$
is completely determined by deciding whether its antiblock(s)
form a flippable set alone, or together with an adjacent (vertical) edge.
\begin{figure}[h]
$$\includegraphics[width=100mm]{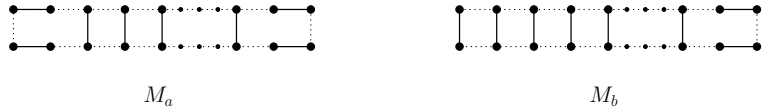}$$
\caption{$M_a$ has no block and exactly two disjoint compatible
  matchings. $M_b$ has one block and exactly one disjoint compatible matching.}
\label{fig:3ex}
\end{figure}

In the drawings in Figure~\ref{fig:3ex}, $\mathbf{\Gamma}$
is a rectangle, and all the edges of the matchings are
either horizontal segments that lie on the lower or on the upper side,
or vertical segments that connect these sides.
Such a representation will be called a \textit{strip drawing}.
Strip drawings are very convenient for representation of certain kinds of matchings,
and they will be used intensively in subsequent sections.
Notice that the fact that horizontal segments lie \textit{on} $\mathbf{\Gamma}$ is inconsistent with our definitions
(in particular, that of the dual graph),
but they can be easily adjusted.
For example, we can treat this drawing as schematic and imagine that the horizontal segments
are in fact slightly curved towards $\mathbf{O(\Gamma)}$.

\subsection{Small components for odd $k$ (Isolated Matchings)}\label{sec:small_odd}

In contrast to the even case,
for each odd $k$
there exist isolated matchings of size $k$.
It is mentioned in~\cite{aichholzer}
that the matchings rotationally equivalent to
$M=\{P_{1}P_{2k}, P_{2}P_{2k-1}, \dots, P_{k}P_{k+1}\}$
are isolated for odd $k$.
In this section we describe all isolated matchings (for the convex case).
Figure~\ref{fig:examples} shows a few examples of isolated matchings -- in fact,
{up to rotation}, these are all isolated matchings of sizes $1$ (a), $3$ (b), $5$ (c, d).

\begin{figure}[h]
$$\includegraphics[width=110mm]{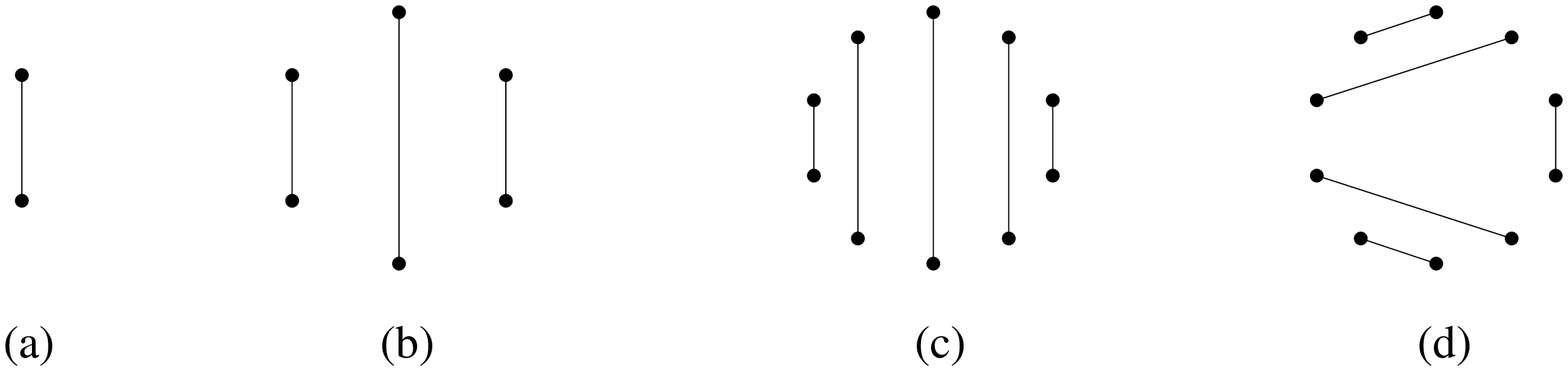}$$
\caption{Examples of isolated matchings.}
\label{fig:examples}
\end{figure}


\noindent\textbf{Definition.}
An \textit{I-matching} is
either a (unique) matching of size $1$,
or a matching of odd size $k \geq 3$
obtained from an I-matching of size $k-2$
by inserting a block in any place.

\begin{theorem}\label{thm:small_odd_struct}
A matching of odd size $k$ is isolated in $\dcm_k$ if and only if it is an I-matching.
\end{theorem}

\begin{proof} Let $M$ be a matching of odd size $k$. For $k=1$ the statement is clear. Assume $k \geq 3$.

If $M$ has no blocks, then it is not isolated by Proposition~\ref{thm:few_blocks}~(1), and it is not an I-matching by definition.

If $M$ has at least one block, the theorem follows 
from Proposition~\ref{thm:block_keeps_degree}
which says that inserting a block does not change the degree.
\end{proof}

We prove several facts about I-matchings to be used later.

\begin{observation}\label{thm:small_odd_two_blocks}
An $I$-matching of size $k\geq 3$ has at least two blocks (which are disjoint for $k \geq 5$).
\end{observation}

\begin{proof}
By Proposition~\ref{thm:few_blocks}, for $k>1$, any matching with at most one block is not isolated.
For $k\geq 4$, two blocks are always disjoint.
\end{proof}

\begin{proposition}\label{thm:isolated_no_antiblock}
If $M$ is an I-matching, then it has no antiblocks.
\end{proposition}

\begin{proof}
The matching of size $1$ clearly has no blocks.
An insertion of a block into a matching without antiblocks
never produces a matching with an antiblock.
\end{proof}

We color the edges of I-matchings in the following way.
Let $M$ be an I-matching of size $k$, and let $e\in M$.
Then $e$ separates $M$ into two (possibly empty) submatchings whose total size is $k-1$.
If both these submatchings are of even size, $e$ will be colored red;
if they are of odd size, $e$ will be colored black.
The edges of $D(M)$ will be colored correspondingly.
See Figure~\ref{fig:small_odd_tree_2}.
The following facts are obvious, or easily seen by induction.

\begin{observation}\label{thm:blue_black}
Let $M$ be an I-matching of size $k$.
\begin{enumerate}
\item The only edge of the matching of size $1$ is red.
\item When a block $K$ is inserted in $M$ so that an I-matching $M+K$ is obtained,
then
the edges of $M+K$ corresponding to those of $M$, preserve their color;
and the edges corresponding to those of $K$ are colored as follows:
the outer edge is black,
and the inner edge is red.
\item The number of red edges is $\ell \left(= \left\lceil \frac{k}{2} \right\rceil\right)$,
and the number of black edges is $\ell-1$.
\item Each face of the dual map of $M$ has exactly one red edge.
Correspondingly, each vertex of $D(M)$ is incident to exactly one red edge.
\end{enumerate}
\end{observation}

\begin{figure}[h]
$$\includegraphics[width=110mm]{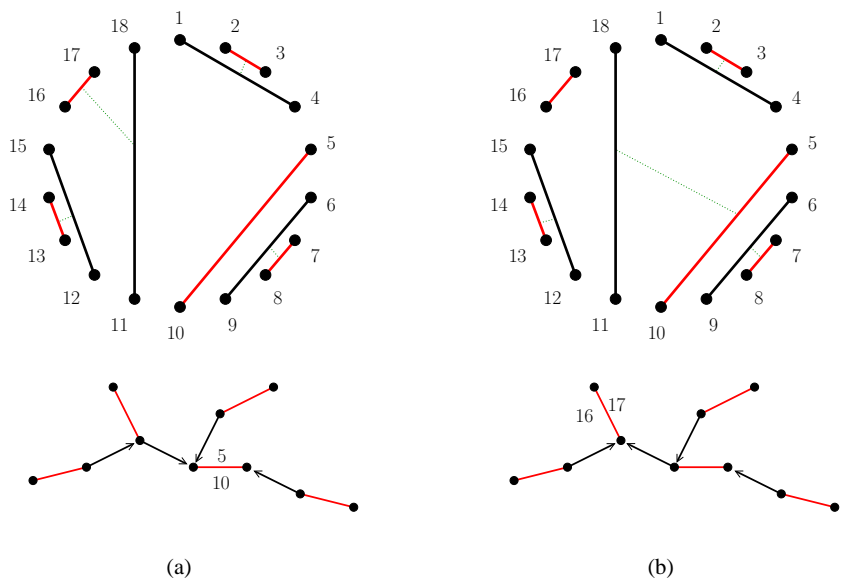}$$
\caption{An I-matching and its dual graph. (a) The root is $P_{5}P_{10}$. (b) The root is $P_{16}P_{17}$. }
\label{fig:small_odd_tree_2}
\end{figure}

According to the definition, in order to construct an I-matching $M$
we start with a matching of size $1$, and insert blocks recursively.
The edge of $M$ corresponding to the initial edge
will be called the \textit{root}.
Pairs of edges corresponding to the members of a block
inserted in some stage of the recursive construction,
will be called \textit{twins}.
However, the same I-matching can be constructed in several ways,
and therefore the root and the twins are not uniquely defined for $M$
but rather depend on the specific construction (a sequence of insertions of blocks).
Referring to a specific construction, we connect twins by green dotted lines
(thus, the root is the only edge not connected in this way to any other edge).
In the dual graph, we draw an arrow on the black edge which points to the point to which it is attached.
See Figure~\ref{fig:small_odd_tree_2}(b) for an example:
in the first drawing the root is $P_5P_{10}$,
in the second drawing it is $P_{16}P_{17}$.
See Figure~\ref{fig:small_odd_tree_2}(b) for an example:
in the first drawing the root is $P_5P_{10}$,
in the second drawing it is $P_{16}P_{17}$
(notice that the order of inserting the blocks can be also chosen in several ways).

\begin{proposition}\label{thm:small_odd_many_roots}
Let $M$ be an I-matching.
\begin{enumerate}
\item For any red edge $e$ of $M$,
there exists a recursive construction of $M$ such that $e$ is the root.
\item For each choice of the root, the pairs of twins are determined uniquely.
\end{enumerate}
\end{proposition}

\begin{proof} For $k=1$ the statements hold trivially. Assume $k \geq 3$.
Let $K$ be a block that does not contain $e$
(existence of such a block is clear for $k=3$,
and follows from Observation~\ref{thm:small_odd_two_blocks} for $k\geq 5$).
\begin{enumerate}
\item
By induction,
there exists a recursive construction of $M-K$ such that the edge corresponding to $e$
is the root.
Upon inserting $K$, $e$ is a root of $M$.
\item
The inner edge of $K$ can be a twin only of the outer edge of $K$.
Then we continue inductively for $M-K$.
\end{enumerate}
\end{proof}



\begin{restatable}{theorem}{smalloddenum}\label{thm:small_odd_enum}
The number of I-matchings of size $k$ is $\frac{1}{\ell}\binom{4\ell-2}{\ell-1}$ (where $ \ell = \left\lceil \frac{k}{2} \right\rceil$).
\end{restatable}

The proof of Theorem~\ref{thm:small_odd_enum} is closely related to that of enumeration of L-matchings that well be introduced in Section~\ref{sec:leaves}.
Therefore, these proofs will be given together (in Section~\ref{sec:leaves_enum}).

\subsection{Leaves}\label{sec:leaves}
In this section we study the matchings that correspond to leaves -- 
that is, vertices of degree $1$ -- in $\dcm_k$ (for both odd and even values of $k$).



\medskip

\noindent\textbf{Definition.}
An \textit{L-matching} is
either a ring of size $2$,
a ring of size $3$,
or a matching of size $k \geq 4$
that can be obtained from an L-matching of size $k-2$
by inserting a block in any place.

\begin{theorem}\label{thm:leaves_struct}
Let $k$ be any natural number.
A matching of size $k$ is a leaf in $\dcm_k$ if and only if it is an L-matching.
\end{theorem}

\begin{proof} For
$k \leq 3$ the statement holds trivially or can be verified directly.
Assume $k \geq 4$.

If $M$ has no blocks, then by Proposition~\ref{thm:few_blocks}~(1) it has at least two neighbors
and thus is not a leaf, and it is not an L-matching by definition.

If $M$ has at least one block, the theorem follows from Proposition~\ref{thm:block_keeps_degree}
which says that inserting a block doesn't change the degree.
\end{proof}

Thus, the recursive construction of L-matchings is very similar to that of I-matchings --
only the basis is different.
We define roots and twins for L-matchings similarly to the case of I-matchings,
with the following difference.
For even $k$,
we do not define root,
and the edges corresponding to the initial pair of edges will be also called twins.
For odd $k$,
the edges corresponding to the initial triple of edges will be called \textit{the root triple}.

\begin{proposition}\label{thm:leaves_roots}
Let $M$ be an L-matching.
\begin{enumerate}
\item For even $k$, the pairs of twins are determined uniquely.
\item For odd $k$, the root triple and the pairs of twins are determined uniquely.
\end{enumerate}
\end{proposition}

\begin{proof} The pairs of twins and (in the odd case) the root triple form a flippable partition.
Thus, the uniqueness follows in both cases from the fact that
any L-matching is disjoint compatible to exactly one matching and,
therefore, it has exactly one flippable partition.
\end{proof}

\subsection{Enumeration of I- and L-matchings}\label{sec:leaves_enum}

Enumeration of I-matchings and L-matchings will be based on the following well-known result
about non-crossing partitions.
A \textit{non-crossing partition} of a set of points in convex position
is a partition of this set into non-empty subsets
whose convex hulls do not intersect
(thus, a non-crossing matching is essentially a non-crossing partition in which all the subsets are of size $2$).

\begin{theorem}[Essentially, a special case of a result by N.\ Fuss from 1791~\cite{fuss}]
\label{thm:fuss}
For $\ell\geq 0$, let $a_\ell$ be the number of non-crossing partitions of a set of $4\ell$ labeled points in convex position into
$\ell$ quadruples ($a_0=1$ by convention).
Let $g(x)=a_0+a_1x+a_2x^2+\dots$ be the corresponding generating function.
Then:
\begin{enumerate}
\item The generating function $g(x)$ satisfies the equation
\begin{equation}\label{eq:gf_quadruples}
g(x) = 1+xg^4(x).
\end{equation}
\item The numbers $a_\ell$ are given by
\begin{equation}\label{eq:number_quadruples}
a_\ell = \frac{1}{3 \ell+1} \binom{4 \ell}{\ell}.
\end{equation}
\end{enumerate}
\end{theorem}

\noindent\textit{Remarks.}
\begin{enumerate}
\item N.\ Fuss proved that for fixed $d \geq 2$,
the number of dissections of a convex $((d-1)\ell + 2)$-gon by its diagonals
into $\ell$ \ \  $(d+1)$-gons is $\frac{1}{(d-1)\ell+1} \binom{d\ell}{\ell}$,
and (essentially) that the corresponding generating function satisfies the equation $g(x) = 1+xg^{d}(x)$.
These numbers are known as Pfaff-Fuss (or Fuss-Catalan) numbers.
For $d=2$, Catalan numbers are obtained.
See~\cite[A062993]{oeis} for this two-parameter array and~\cite{brown} for a historical note on the topic.
It is easy to see that the two structures --
diagonal dissections of a convex $((d-1)\ell + 2)$-gon into $\ell$ \ \  $(d+1)$-gons
\textit{and}
non-crossing partitions of $d \ell $ points in convex position into $\ell$ sets of size $d$, --
have the same recursive structure
(see~\cite[Exercise $6.19$ (a) and (n)]{stanley} for the case of $d=2$).
Thus, $a_\ell$ are Pfaff-Fuss numbers with $d=4$.
\item Eq.~\eqref{eq:number_quadruples}
-- rather in the form $\frac{1}{\ell} \binom{4 \ell}{\ell-1}$ for $\ell\geq 1$ --
follows from Eq.~\eqref{eq:gf_quadruples} by the Lagrange inversion formula~\cite[Theorem 5.4.2]{stanley}.
Indeed, Eq.~\eqref{eq:gf_quadruples} is equivalent to
$x = \frac{\tilde{g}(x)}{(\tilde{g}(x)+1)^4}$ where $\tilde{g}(x) = g(x)-1$.
Therefore, if, following the notation as in the reference above,
we take $F(x) = \frac{x}{(x+1)^4}$,
or, equivalently, $G(x)=(x+1)^4$, and $k=1$,\footnote{
This $k$ from the statement of the Lagrange inversion formula in~\cite{stanley} is of course different from $k$ as we use it in this paper.\label{fn:k}}
we obtain $
a_\ell=[x^\ell] \tilde{g}(x) =
\frac{1}{\ell} [x^{\ell-1}] G^{\ell}(x) =
\frac{1}{\ell} [x^{\ell-1}] (x+1)^{4\ell} = \frac{1}{\ell} \binom{4 \ell}{\ell-1}$.
\end{enumerate}

\smalloddenum*

\begin{theorem}\label{thm:leaves_enum}\mbox{}
\begin{enumerate}
\item
For odd $k $, the number of L-matchings of size $k$ is $\frac{2}{3} \frac{\ell-1}{\ell} \binom{4 \ell - 2}{\ell - 1} $
(where $ \ell = \left\lceil \frac{k}{2} \right\rceil$).
\item
For even $k $, the number of L-matchings of size $k$ is $\frac{\ell+1}{3 \ell+1} \binom{4 \ell}{\ell}$
(where $ \ell = \left\lceil \frac{k}{2} \right\rceil$).
\end{enumerate}
\end{theorem}

\begin{proof} It will be convenient to prove
first Theorem~\ref{thm:leaves_enum}~(2),
then Theorem~\ref{thm:small_odd_enum},
and finally Theorem~\ref{thm:leaves_enum}~(1).

A matching $M$ and a non-crossing partition $T$ of $X_{2k}$ \textit{fit} each other
if every edge of $M$ connects two points that belong to the same set of the partition $T$.

\smallskip

\noindent\textit{Proof of Theorem~\ref{thm:leaves_enum}~(2).}
Let $M$ be an L-matching of even size $k$.
We saw in Proposition~\ref{thm:leaves_roots} that the edges of $M$
can be partitioned into pairs of twins in a unique way.
Replace each pair of twins by a quadruple of points.
In this way we obtain a (unique) non-crossing partition of $X_{2k}$ into $\ell$ quadruples that fits $M$.

Let $T$ be any non-crossing partition of $X_{2k}$ into $\ell$ quadruples.
We show that there are exactly $\ell+1$ L-matchings that fit $T$.
For $k=2$ ($\ell=1$) there are $2$ L-matchings, both fitting the (unique) non-crossing partition into quadruples.
For $k \geq 4$ ($\ell \geq 2$) we proceed by induction as follows.

Let $s$ be any quadruple of $T$ that consists of four consecutive points $P_i, P_{i+1}, P_{i+2}, P_{i+3}$.
(Such a quadruple will be called an \textit{ear}. Each non-crossing partition
with at least two parts has at least two ears.)
For each L-matching of size $k-2$ that fits $T \setminus \{s\}$,
we can connect $P_i$ with $P_{i+3}$ and $P_{i+1}$ with $P_{i+2}$.
This is inserting a block, and, thus, an L-matching of size $k$ is obtained.
By induction, the number of matchings that we obtain in this way is $\ell$.

In order to obtain one more matching, we connect first
$P_i$ with $P_{i+1}$ and $P_{i+2}$ with $P_{i+3}$.
We show now that this can be completed to an L-matching in exactly one way.
Namely, let $s'$ be any quadruple of $T$ ($s'\neq s$).
Suppose that the points of $s'$ are $P_{\alpha}, P_{\beta}, P_{\gamma}, P_{\delta}$
so that the cyclic order of the labels of the points of $S \cup S'$ satisfies
$i+4 \prec \alpha \prec \beta \prec \gamma \prec \delta \prec i $.
Then we must connect $P_{\alpha}$ with $P_{\delta}$ and $P_{\beta}$ with $P_{\gamma}$.
Indeed, if we do that for each quadruple, an L-matching is obtained.
In order to see that, erase an ear different from $s$.
In this way a block is deleted from a matching, and then the induction applies.
On the other hand, if in some $s'$ we connect $P_{\alpha}$ with $P_{\beta}$ and $P_{\gamma}$ with $P_{\delta}$,
then we have two quadruples of $T$ that contain a flippable pair
and in both (with respect to the order of their union) the first point is connected to the second,
and the third to the fourth.
It is easy to see from the definition that this never happens in L-matchings.

To summarize:
by Theorem~\ref{thm:fuss}, there are
$\frac{1}{3 \ell+1} \binom{4 \ell}{\ell}$
non-crossing partitions of $X_{2k}$ into $\ell$ quadruples,
each such partition fits $\ell+1$ L-matchings,
and each L-matching is obtained in this way exactly once.
Therefore, the number of L-matchings of size $k$ is $\frac{\ell+1}{3 \ell+1} \binom{4 \ell}{\ell}$.

\smallskip

\noindent\textit{Proof of Theorem~\ref{thm:small_odd_enum}.}
First, each I-matching $M$ has exactly one red edge
$e=P_iP_j$ ($i<j$) such that
all other edges of $M$
either connect two points from the set $\{1, 2, \dots, i-1\}$ (\textit{appear before $e$}),
or two points from the set $\{i+1, i+2, \dots, j-1\}$ (\textit{appear inside $e$}),
or two points from the set $\{j+1, j+2, \dots, 2k\}$ (\textit{appear after $e$});
such an edge will be called \textit{the special red edge}.
Indeed, this holds trivially for the matching of size $1$,
and this remains true when a block is inserted:
if a block is inserted between $P_\alpha$ and $P_{\alpha+1}$ where $1 \leq \alpha \leq 2k-1$,
then (only) the edge corresponding to the old special red edge is special;
and if a block is inserted between $P_{2k}$ and $P_{1}$,
then the red edge of this block becomes the special one.

Let $M$ be an I-matching and let $e=P_iP_j$ be its special red edge.
By Proposition~\ref{thm:small_odd_many_roots}, there exists a recursive construction of $M$ such that $e$ is the root.
Replace all the pairs of edges that were inserted as blocks at some step of this construction by quadruples.
Then we have three non-crossing partitions of the corresponding sets of points into quadruples:
one before $e$, one inside $e$, one after $e$.
On the other hand, for each such partition, there is only one way to connect points of each quadruples by two edges
in order to obtain an I-matching.
Namely, for a quadruple $P_{\alpha}, P_{\beta}, P_{\gamma}, P_{\delta}$ with $\alpha < \beta < \gamma < \delta$
we must connect $P_{\alpha}$ with $P_{\delta}$ and $P_{\beta}$ with $P_{\gamma}$.
The proof is similar to that above: the points of an ear must be connected in this way
(otherwise the conclusion of Proposition~\ref{thm:blue_black} (3) is not satisfied),
and then induction applies.

Thus, three non-crossing partitions of points before, inside, and after $e$ into quadruples
determine uniquely an I-matching.
It follows that the generating function for the number of such matchings is $xg^3(x)$,
where $g(x)$ is the function from Theorem~\ref{thm:fuss}.
In order to calculate its coefficients,
we use the general form of the Lagrange inversion formula~\cite[Corollary 5.4.3]{stanley}
with $G(x)=(x+1)^4$, $H(x)=(x+1)^3$ (so that $g^3(x) 
= H(\tilde g(x) )$),
and $k=3$.\footnote{The same remark as in footnote~\ref{fn:k} applies.}
We obtain
\[ [x^\ell] x g^3(x) =
[x^{\ell-1}]  g^3(x) =
[x^{\ell-2}] \frac{1}{\ell-1} H'(x) G^{\ell-1}(x) =
\frac{3}{\ell-1} [x^{\ell-2}] (x+1)^{4\ell-2} =
\frac{3}{\ell-1} \binom{4\ell-2}{\ell-2},\]
which is equal to $\frac{1}{\ell} \binom{4 \ell - 2}{\ell - 1} $ for $\ell>1$.

\smallskip

\noindent \textit{Remark.}
This sequence of numbers is~\cite[A006632]{oeis}, where it appears with a reference to a paper by H.~N.~Finucan~\cite{finucan}.
In that paper, it counts the number of nested systems (``stackings'') of $\ell$ folders with $3$ compartments
such that exactly one folder is outer (``visible'').
There is a very simple bijection between two structures, see Figure~\ref{fig:folders2_new_2} for an example:
pairs of twins are converted into $3$-compartment folders;
the special red edge forms a pair with the outer part of $\mathbf{\Gamma}$,
and it is converted to the outer folder.
\begin{figure}[h]
$$\includegraphics[width=110mm]{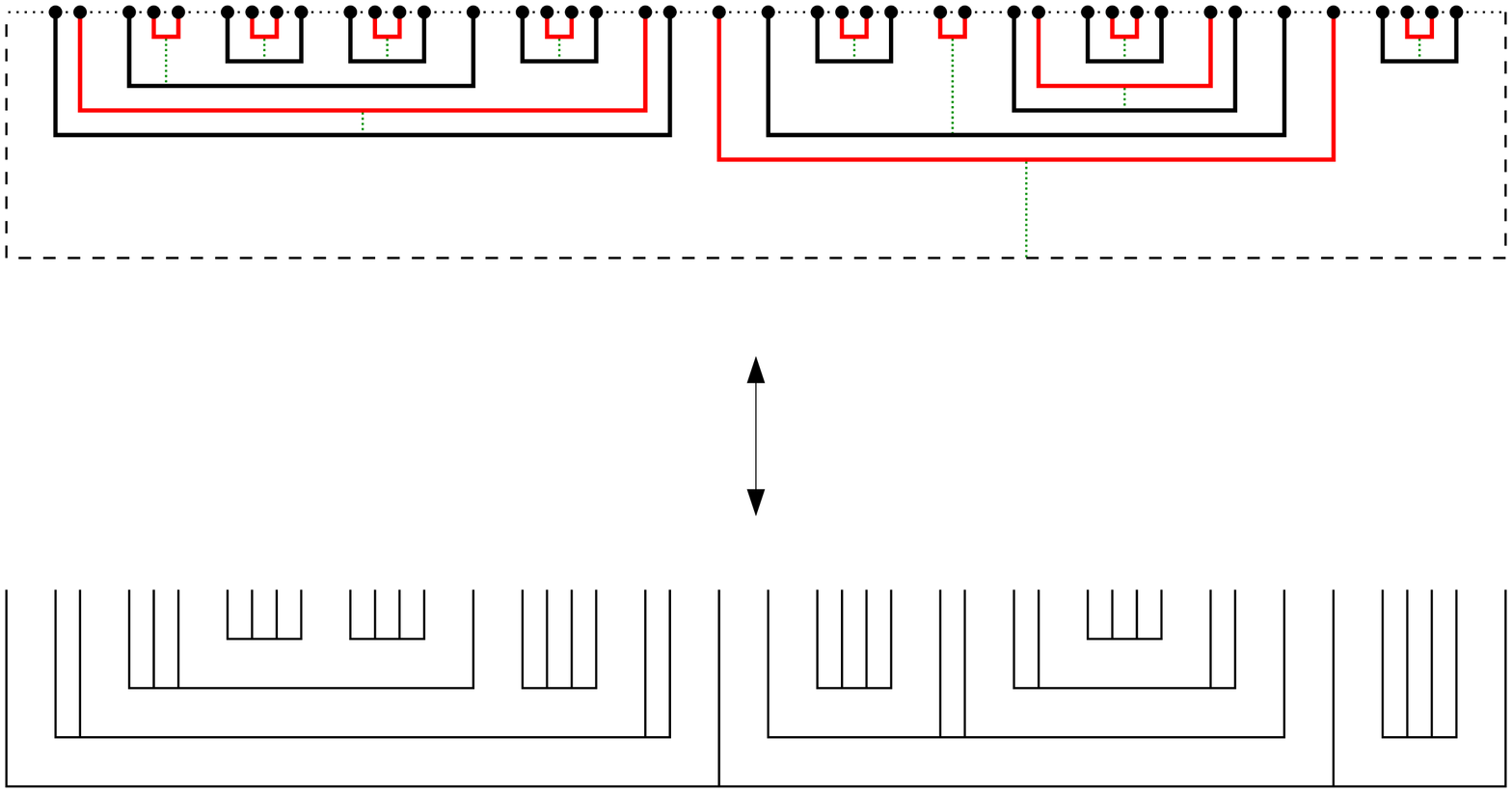}$$
\caption{An example illustrating the bijection between I-matchings of size $k=2\ell-1$ and
stackings of $\ell$ $3$-folders with only one outer folder.}
\label{fig:folders2_new_2}
\end{figure}

\smallskip

\noindent\textit{Proof of Theorem~\ref{thm:leaves_enum}~(1).}
The proof will be based on the previous one (notice the similarity of the expressions in these two theorems).
Essentially, we describe a way to convert I-matchings into L-matchings of odd size,
and take care of multiplicities.

Let $M$ be an I-matching of size $k \geq 3$.
Each black edge belongs to two faces, and, by Observation~\ref{thm:blue_black} (4),
each of these faces has exactly one red edge.
Such a triple of edges -- a black edge $e$ and the red edges incident to the faces incident to $e$ --
will be called a RBR-triple.\footnote{\textit{RBR} stands for red-black-red.}
By Observation~\ref{thm:blue_black} (3), there are $\ell-1$ black edges in $M$;
therefore, there are also $\ell-1$ RBR-triples.
Therefore, there are $\frac{\ell-1}{\ell}\binom{4\ell-2}{\ell-1}$
I-matchings of size $k$ with a marked RBR-triple.

Suppose that the endpoints of the edges that belong to an RBR-triple
are
(according to the cyclic order)
$Q_{1}, Q_{2}, Q_{3}, Q_{4}, Q_{5}, Q_{6}$.
Then the RBR-triple can be one of the following:
$\{Q_{1}Q_{2}, Q_{3}Q_{6}, Q_{4}Q_{5}\}$,
$\{Q_{1}Q_{4}, Q_{2}Q_{3}, Q_{5}Q_{6}\}$, or
$\{Q_{1}Q_{6}, Q_{2}Q_{5}, Q_{3}Q_{4}\}$.
It is easy to see that if we replace these edges by
either
$\{Q_{1}Q_{2}, Q_{3}Q_{4}, Q_{5}Q_{6}\}$
or
$\{Q_{2}Q_{3}, Q_{4}Q_{5}, Q_{6}Q_{1}\}$,
an L-matching is obtained.
Thus, we have obtained $2 \frac{\ell-1}{\ell}\binom{4\ell-2}{\ell-1}$ L-matchings.

However, each L-matching is obtained in this way exactly three times.
Indeed, by Proposition~\ref{thm:leaves_roots} (2), the root triple of an L-matching is determined uniquely.
It can be replaced by a RBR-triple in three ways, each of them producing an I-matching.
Therefore, the number of L-matchings of size $k$
(for odd $k$)
is $\frac{2}{3} \frac{\ell-1}{\ell} \binom{4 \ell - 2}{\ell - 1} $.
\end{proof}


\subsection{Strip Drawings and DB-components}\label{sec:strip}

In the following sections, we shall frequently use a special way to draw matchings -- \emph{strip drawings},
that were already used in the end of Section~\ref{sec:small_general}.
In such a drawing $\mathbf{\Gamma}$ is an axis-aligned rectangle $\mathbf{R}$, and all the points of $X_{2k}$ lie
on its horizontal sides (the lower side will be denoted by $\mathbf{L}$, the upper by $\mathbf{U}$).
The edges that connect a point from $\mathbf{L}$ with a point of $\mathbf{U}$ will be represented by vertical segments;
such edges will be called \textit{D-edges}.
In some cases, in order to achieve a drawing in which all the D-edges are vertical,
we'll move some points of $X_{2k}$ along $\mathbf{L}$ or $\mathbf{U}$.
If a D-edge connects the leftmost (respectively, the rightmost) points of $X_{2k}$
on $\mathbf{L}$ and on $\mathbf{U}$, we will assume that it lies on the left (respectively, the right)
side of $\mathbf{R}$.
The edges that connect neighboring points of $\mathbf{L}$ or of $\mathbf{U}$ will
be represented by horizontal segments that lie on $\mathbf{\Gamma}$;
such edges will be called \textit{B-edges}.\footnote{
\textit{D} and \textit{B} stand for ``diagonal'' and ''boundary'', since a B-edge is always a boundary edge,
and a D-edge is \textit{usually} a diagonal edge
(the exceptional situation is when it connects the leftmost or the rightmost points of $\mathbf{L}$ and $\mathbf{U}$).}
Edges that connect non-neighboring points of $\mathbf{L}$ or of $\mathbf{U}$ will be represented, as usually, by Jordan curves inside $\mathbf{O(\Gamma)}$.
The index of the leftmost point of $\mathbf{U}$ will be denoted by $z$,
and, as agreed earlier, the points are labeled cyclically clockwise.

Obviously, each matching can be represented by a strip drawing,
but we shall use them only for certain classes of matchings,
when such drawings can be made especially simple and clear.
As mentioned earlier, the fact that all the boundary edges
lie on $\mathbf{\Gamma}$ is inconsistent with our original definitions.
In particular, as a planar map, such a drawing ``looses'' all the boundary faces
(therefore it will be called a \textit{reduced map}).
However, strip drawings are very useful due to the following fact.
As mentioned above, a flippable set is a subset of the set of edges that belong to the same face.
On the other hand, a flippable set is always of size at least $2$.
Thus, reduced maps have no faces that cannot contribute to a flippable partition,
and, thus, the candidates for flippable sets will be clearly seen.


An \textit{element} in a strip drawing is a subset of edges 
that can be separated from other edges by straight lines.
We distinguish the following kinds of elements; they will be used later for describing of certain kinds of matchings.
Refer to Figure~\ref{fig:db_etc_flip}.
A \emph{DB-element} in an element of size $2$
that consists of a D-edge $d$ and a B-edge $b$.
There are four kinds of DB-elements, distinguished by their \textit{direction} and \textit{position} as follows.
The direction is $\mathrm{R}$ if $b$ is to the right of $d$, $\mathrm{L}$ if $b$ is to the left of $d$.
The position is $-$ if $b$ lies on $\mathbf{L}$, and $+$ if $b$ lies on $\mathbf{U}$.
A \textit{DBD-element} is an element of size $3$
that consists of two D-edges $d_1, d_2$, and one B-edge $b$ between them.
The position of a DBD-element is $-$ (respectively, $+$) if $b$ lies on $\mathbf{L}$ (respectively, on $\mathbf{U}$).
A \textit{B$^{2+1}$-element} is an element of size $3$
that consists of three B-edges: two on $\mathbf{L}$ and one on $\mathbf{U}$ (then its position is $-$), or vice versa (then its position is $+$).
An \textit{EDB-element} is an element of size $4$
that consists of three B-edges forming a B$^{2+1}$-element and a D-edge to the left or to the right of them.
The direction of an EDB-element is $\mathrm{R}$ (respectively, $\mathrm{L}$)
if the B-edges are to the right (respectively, to the left) of the D-edge;
its position agrees with that of the B$^{2+1}$ element.
Notice that DB-, EDB-, DBD- and B$^{2+1}$-elements are always flippable sets.
The next observation summarizes the effect of flipping these elements.

\begin{observation}\label{thm:db_flip}\mbox{}
\begin{enumerate}
\item The set obtained from a DB-element by flipping is a DB-element with the same position and different direction.
\item The set obtained from an EDB-element by flipping is an EDB-element with the same position and different direction.
\item The set obtained from a DBD-element by flipping is a B$^{2+1}$-element with the same position, and vice versa.
\end{enumerate}
\end{observation}

See Figure~\ref{fig:db_etc_flip} for illustration.
Notice that in some cases we modify the point set in order to draw a D-edge as a vertical segment.
On the first strip, given elements are shown;
on the second, the elements obtained from them by flipping;
on the third, they are shown after modifying the point set.
\begin{figure}[h]
$$\includegraphics[width=130mm]{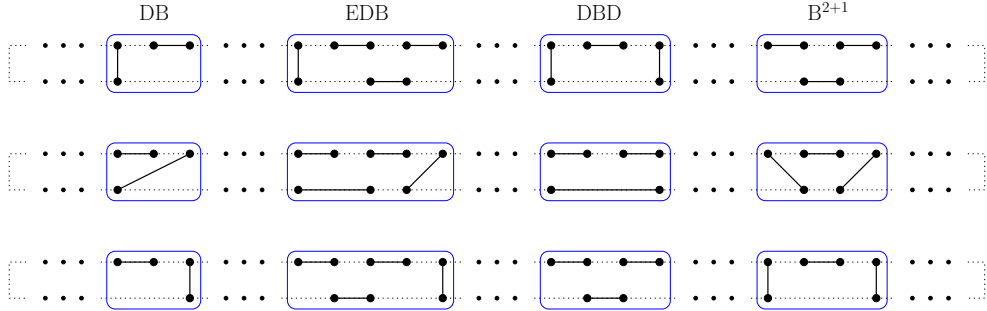}$$
\caption{DB-, EDB-, DBD-, and B$^{2+1}$-elements, and flipping them.}
\label{fig:db_etc_flip}
\end{figure}

The structure of some simple matchings can be partially described by their
\textit{pattern} --
a sequence of elements of these types
(to be read from left to right).
For example, we say that a strip drawing has {pattern} DBDB$^{2+1}$D if it consists of three D-edges $d_1, d_2, d_3$,
a B-edge between $d_1$ and $d_2$, and a B$^{2+1}$-element between $d_2$ and $d_3$.
Notice that the pattern does not determine a drawing uniquely since the labeling of points
and the position of B-edges is not indicated.



\subsection{Small components for even $k$ (Pairs)}\label{sec:small_even}


By Proposition~\ref{thm:even_never_isolated}, a matching of even size is never isolated.
As we shall show now, for any even $k$ there are matchings of size $k$ that belong to \emph{pairs} --
connected components of size $2$. Thus, we next define a family of
matchings and prove that they indeed form the small components of $\dcm_k$ for even values of $k$.

\medskip

\noindent\textbf{Definition.}
Let $k$ be an even number.
A \textit{DB-matching} of size $k$ is a matching
that can be represented by a strip drawing
with pattern $\mathrm{DBDB}\dots\mathrm{DB}$ --
that is, consists of $\ell \left( = \left\lceil \frac{k}{2} \right\rceil \right) $
R-directed DB-elements.
\medskip

A drawing as in this definition will be the \textit{standard drawing} for a DB-matching.
If instead of R-directed DB-elements we have L-directed DB-elements,
this is an \textit{upside-down drawing} of a DB-matching;
the standard one can be obtained from it by $180^\circ$ rotation.
The edges of the $i$th (from left to right) DB-element in the standard drawing of a DB-matching will be denoted by $d_i, b_i$.
The map of $M$ has $\ell$ inner faces and $\ell+1$ boundary faces.
The inner faces will be denoted by $D_1, D_2, \dots, D_{\ell}$:
for $1 \leq i \leq \ell-1$, $D_i$ is the face whose edges are $d_i, b_i, d_{i+1}$;
$D_\ell$ is the face whose edges are $d_\ell, b_\ell$.
The boundary faces will be denoted by $B_0, B_1, \dots, B_{\ell}$:
$B_0$ is the face whose only edge is $d_1$;
for $1 \leq i \leq \ell$, $B_i$ is the face whose only edge is $b_i$.

In a DB-matching of size $k \geq 4$, $\{d_1, b_1\}$ is an antiblock,
and $\{d_\ell, b_\ell\}$ is a block,
and there are no other separated pairs.
Therefore, the position
($-$ or $+$) of these extremal DB-elements can be chosen arbitrarily:
changing the position of $\{d_\ell, b_\ell\}$ does not change the matching,
and changing the position of $\{d_1, b_1\}$ results in a rotationally isomorphic matching.
For $k \geq 4$,
we shall always draw the antiblock as a DB-element of type $\zr +$,
and the block as a DB-element of type $\zr -$.
Different choices of position in all other DB-elements produce rotationally non-equivalent matchings.
Their positions will be encoded by a $\{-, + \}$-sequence $\chi = (x_1, x_2, \dots, x_{\ell-2})$,
where $x_i$ is the position of the $(i+1)$st DB-element.
The DB-matching of size $k$ with specified $\chi$ and $z$
(the label of the leftmost point on $\mathbf{U}$)
will be denoted by $\db(k, \chi, z)$.\footnote{Note that $k$ is determined by the length of $\chi$ and, therefore,
can be omitted. However, we find it convenient to include it in our notation.}

The dual trees of DB-matchings have the following structure
(we denote the vertices of $D(M)$ identically to the corresponding faces of the map of $M$):
There is a path $B_0D_1D_2\dots D_{\ell}$
(imagined as consisting of horizontal edges so that $B_0$ is on the left and $D_{\ell}$ is on the right);
and for each $i$, $1 \leq i \leq \ell$, a leaf $B_i$ is attached to $D_i$.
As explained above, by convention $B_1$ is attached to $D_1$ above the path,
and $B_\ell$ is attached to $D_\ell$ below the path;
and for $2 \leq i \leq \ell-1$,
$B_i$ can be attached to $D_i$ in two ways: either below or above the path.
See Figure~\ref{fig:small_even_tree_new_2}:
(a) shows the matching $\db(14, -++-+, 1)$ represented by its standard strip drawing;
(b) shows its dual tree;
(c) shows the general structure of the dual tree of DB-matchings
(dashed edges $D_iB_i$, $2 \leq i \leq \ell-1$, indicate that each of them can be either below or above the path $B_0 D_1\dots D_\ell$).

\begin{figure}[h]
$$\includegraphics[width=120mm]{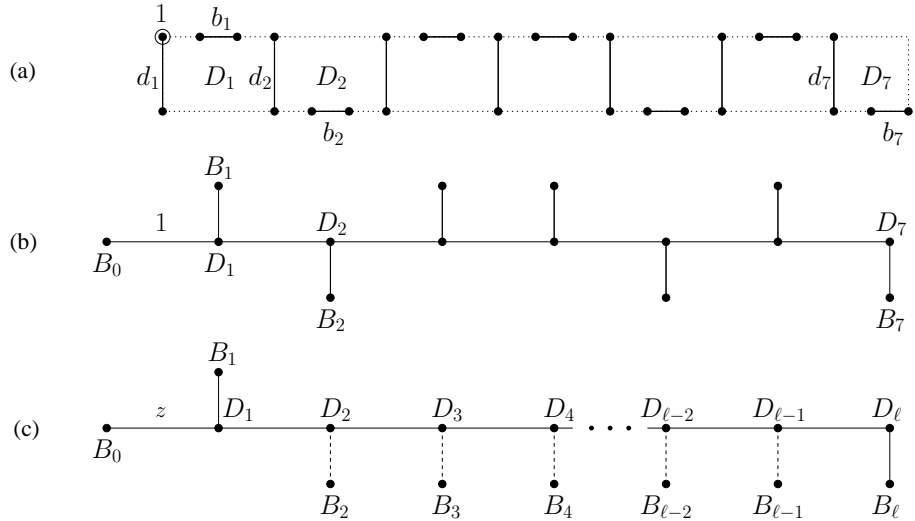}$$
\caption{
(a) The matching $\db(14, -++-+, 1)$.
(b) The dual tree of $\db(14, -++-+, 1)$.
(c) The general structure of the dual tree of DB-matchings.
}
\label{fig:small_even_tree_new_2}
\end{figure}

For a $\{-, +\}$-sequence $\chi$,
we denote by $\chi'$ the sequence obtained from $\chi$
by reversing and changing all the components,
and we denote $\delta(\chi) = \#_\chi (+) - \#_\chi (-)$.
For example, for $\chi=(++-++--+)$
we have $\chi'=(-++--+--)$
and $\delta(\chi)=2$.

\begin{theorem}\label{thm:small_even_struct}
Let $k$ be an even number.
A matching of size $k$ belongs to a pair in $\dcm_{k}$
if and only if
it is a DB-matching.
\end{theorem}

\begin{proof} For $k=2$ the statement is obvious. Thus, we assume $k \geq 4$.

\smallskip

$[\Leftarrow]$
Assume that $M$ is a DB-matching of size $k$.
First we show that it is an L-matching.
The rightmost DB-element of $M$, $K=\{ d_\ell, b_\ell \}$, is a block.
The matching $M-K$ is also a DB-matching, and, therefore it is an L-matching by induction.
Therefore, $M$ is also an L-matching, that is, it has degree $1$ in $\dcm_k$.
Its only flippable partition consists of the DB-elements $\{ d_i, b_i \}$.

Denote the only neighbor of $M$ by $M'$.
By Observation~\ref{thm:db_flip}, $M'$ is obtained from $M$
by replacing each of its DB-elements by the L-directed DB-element of the same position.
This means that $M'$, drawn on the same strip drawing,
is also a DB-matching, but drawn upside down.
In order to obtain its standard representation, we rotate the drawing.
$\chi$ is replaced then by $\chi'$,
and $z$ by the label of the rightmost point on $\mathbf{L}$
in the standard drawing of $M$, which is
$z' = z + k + \delta(\chi)$.\footnote{Indeed, 
let $u=\#_\chi (+)$, $d=\#_\chi (-)$.
Then
the number of points on $\mathbf{U}$
is $3u+d =
 2(u+d) + (u-d)= k + \delta(\chi)$.}
Thus, we obtain $M'=\db(k, \chi', z')$.
See Figure~\ref{fig:small_even_ex_new} for an illustration
(the flippable sets are marked by blue color; the asterisk indicates an upside down drawing).
\begin{figure}[h]
$$\includegraphics[width=130mm]{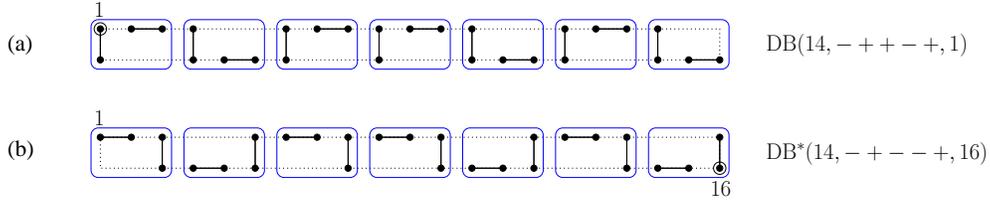}$$
\caption{Two DB-matchings forming a pair: (a) $\db(14, -++-+, 1)$, (b) $\db(14, -+--+, 16)$ (drawn upside down).}
\label{fig:small_even_ex_new}
\end{figure}

Since $M'$ is also a DB-matching, it is adjacent to only one matching, namely, to $M$.
Thus, $M$ and $M'$ form a pair in $\dcm_k$.

\smallskip

$[\Rightarrow]$
Assume that $M$ belongs to a pair.
$M$ has at least one block, as otherwise it is adjacent to at least two distinct matchings by Proposition~\ref{thm:few_blocks}~(1).
Fix a block $K$ in $M$, and denote $N=M-K$.
If $N$ is not a DB-matching, then, by induction and by Proposition~\ref{thm:even_never_isolated},
it is connected (by a path) to at least two matchings.
Then $M$ is connected (by a path) to at least two matchings by Proposition~\ref{thm:block_anti_connected}, and this is a contradiction.

Now assume that $N$ is a DB-matching (of size $k-2$).
We shall see that either $M$ is a DB-matching,
or $M$ can be decomposed in a different way, $M=L+P$,
where $P$ is a separated pair, and $L$ is \textbf{not} a DB-matching
(which will be shown by indicating an element which never occurs in DB-matchings).
In the former case this completes the proof,
in the latter case 
we obtain a contradiction as above (with $L$ in role of $N$ and $P$ in the role of $K$).

Consider the dual tree of $N$.
Then $D(K)$, the part that corresponds to $K$,
is a $2$-branch attached to $D(N)$ in some point (see Figure~\ref{fig:even_cases_new_1}).
Label the points of $D(N)$ in accordance to our usual notation,
as in Figure~\ref{fig:small_even_tree_new_2}
(notice that it consists of $\ell-1$ rather than of $\ell$ DB-elements).
Now we have the following subcases.
\begin{enumerate}
\renewcommand{\labelenumi}{(\alph{enumi})}
\item \textbf{$D(K)$ is attached to $D(N)$ at $B_i$, $0 \leq i \leq \ell-2$.}
Let $P$ be the block $D_{\ell-2}D_{\ell-1}B_{\ell-1}$,\footnote{For the sake of brevity,
we write ``the block/the antiblock $ABC$''
instead of
``the block/the antiblock
corresponding to
the $2$-branch/the V-shape $ABC$''.
} and let $L=M-P$.
Then $D(L)$ has a $3$-branch,
and, therefore, $L$ is not a DB-matching.
\item \textbf{$D(K)$ is attached to $D(N)$ at $D_i$, $1 \leq i \leq \ell-3$.}
Let $P$ be the block $D_{\ell-2}D_{\ell-1}B_{\ell-1}$, and let $L=M-P$.
Then $D(L)$ has a vertex of degree $4$, and, therefore, $L$ is not a DB-matching.
\item \textbf{$D(K)$ is attached to $D(N)$ at $D_{\ell-2}$.}
Let $P$ be the antiblock $B_0D_1B_1$, and let $L=M-P$.
Then $D(L)$ has a vertex of degree $4$, and, therefore, $L$ is not a DB-matching.
\item \textbf{$D(K)$ is attached to $D(N)$ at $D_{\ell-1}$.}
Then $M$ is a DB-matching.
\item \textbf{$D(K)$ is attached to $D(N)$ at $B_{\ell-1}$.}
Let $P$ be the antiblock $B_0D_1B_1$, and let $L=M-P$.
Then $D(L)$ has a $4$-chain, and, therefore, $L$ is not a DB-matching.
\end{enumerate}

These cases are shown in Figure~\ref{fig:even_cases_new_1}.
$D(K)$ is shown by green when $M$ is a DB-matching,
and by blue when a contradiction is obtained.
In this latter case, the element corresponding to $P$ is marked by red.
The point where $D(K)$ is attached to $D(N)$ is marked by a circle.
\begin{figure}[h]
$$\includegraphics[width=160mm]{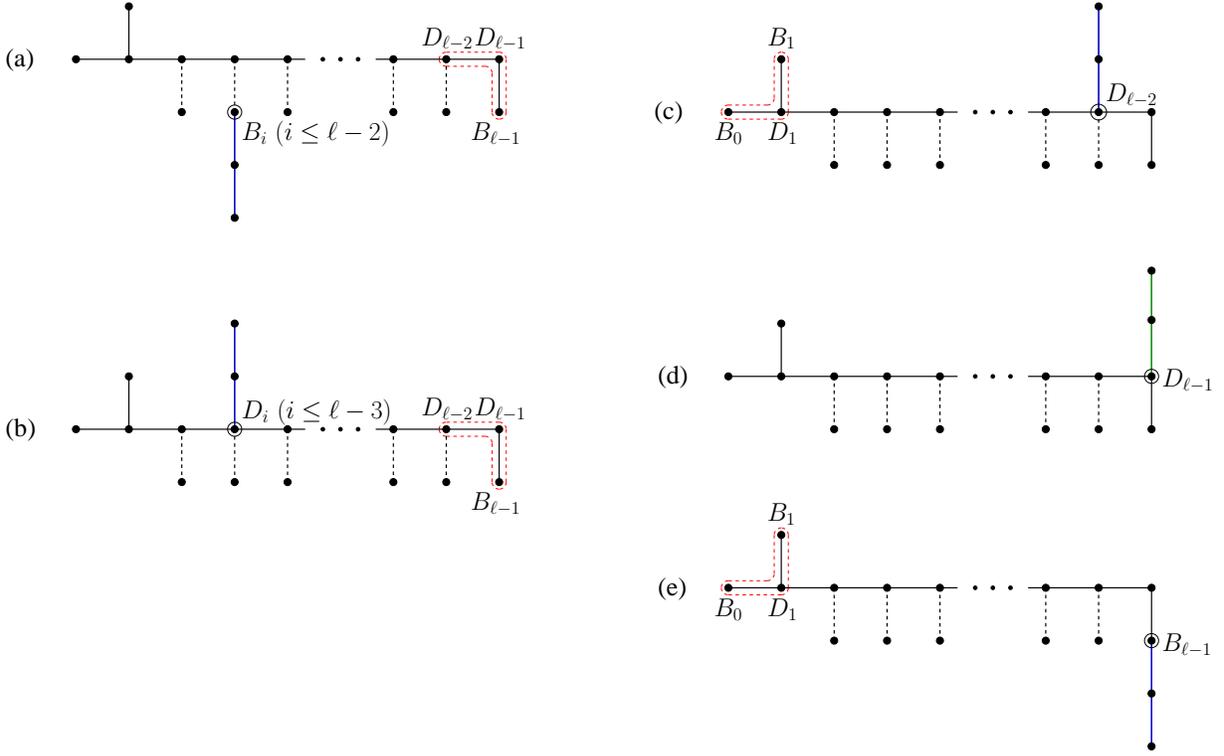}$$
\caption{Illustration to the proof of Theorem~\ref{thm:small_even_struct}.}
\label{fig:even_cases_new_1}
\end{figure}
\end{proof}



\begin{theorem}\label{thm:bd_enum}
The number of DB-matchings of size $k$ is $\ell \cdot  2^\ell$.
\end{theorem}
\begin{proof}
For a DB-matching of size $k$,
$\chi$ can be chosen in $2^{\ell-2}$ ways,
and $z$ in $2k=4\ell$ ways.
Since the structure of a DB-matching has no non-trivial symmetries,
each DB-matching is counted in this way exactly once.
Therefore, there are $2^{\ell-2} \cdot 4\ell = \ell \cdot 2^\ell$ DB-matchings.
\end{proof}

The number of small components in $\dcm_k$ is obtained now immediately.

\begin{corollary}\label{thm:pairs_enum}
The number of small components in $\dcm_k$ is $\ell \cdot  2^{\ell-1}$.
\end{corollary}

\section{Medium components}\label{sec:middle}
\subsection{Medium components for odd $k$}\label{sec:middle_odd}

\noindent\textbf{Definition.}
Let $k \geq 3$ be an odd number.
A \textit{DBD-matching} of size $k$ is a matching
that can be represented by a strip drawing
with pattern $\mathrm{DBDB}\dots\mathrm{DBD}$.
In other words, its strip drawing can be obtained from the standard strip drawing of a DB-matching of size $k-1$
by adding one more D-element that connects the rightmost points of $\zbl$ and $\zbu$.

\medskip

For DBD-matchings, we adopt the notations and the conventions developed for DB-matchings
and their standard drawings.
One difference is that this time the edges of (the rightmost) face $D_{\ell-1}$ are $d_{\ell-1}, b_{\ell-1}, d_{\ell}$.
Similarly to DB-matchings, 
it will be assumed without loss of generality that $b_1$ lies on $\zbu$, and $b_{\ell-1}$ lies on $\zbl$,
and the position of other $b_i$s will be specified by
a $\{-,+\}$-sequence $\chi$ (which is now of length $\ell-3$).
A DBD-matching with specified $\chi$ and $z$ will be denoted by $\zsc(k, \chi, z)$.
Notice, however, that due to a symmetry of the structure each DBD-matching is represented twice in this form:
$\zsc(k, \chi, z) = \zsc(k, \chi', z')$
(or, more precisely,
the standard drawing of $\zsc(k, \chi, z)$
is the upside down drawing of $ \zsc(k, \chi', z')$),
where $\chi'$ and $z'$ are defined as for DB-matchings.
See Figure~\ref{fig:middle_odd_tree_new}:
(a) shows the matching $\zsc(15, ++--+, 1)$ represented by a standard strip drawing
(this matching is also $\zsc(15, -++--, 17)$ drawn upside down),
(b) shows the dual tree of $\zsc(15, ++--+, 1)$,
(c) shows the general structure of the dual tree of DBD-matchings.

\begin{figure}[h]
$$\includegraphics[width=140mm]{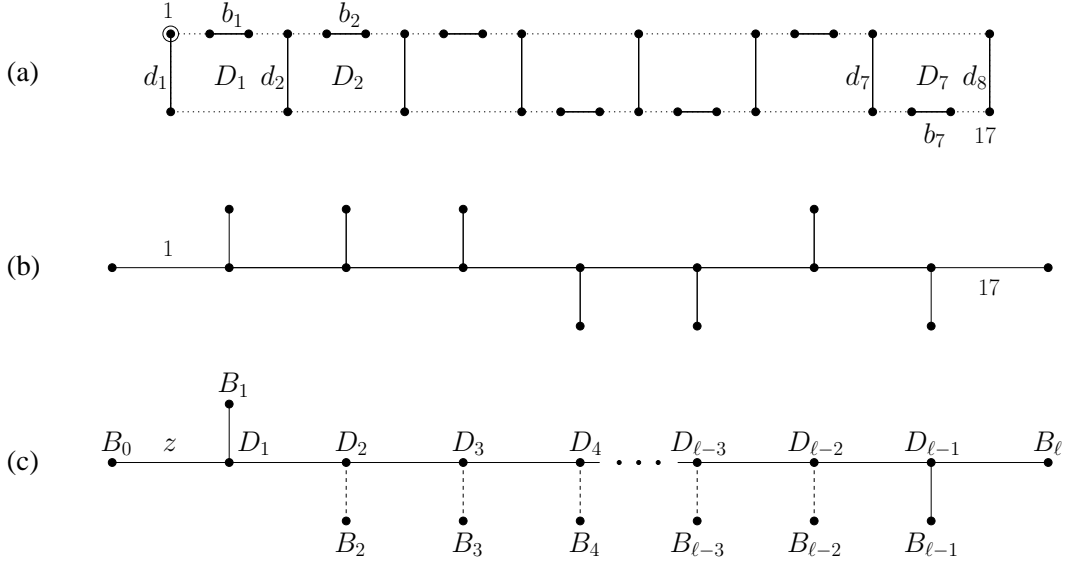}$$
\caption{DBD-matchings:
(a) $\zsc(15, ++--+, 1)$;
(b) The dual tree of $\zsc(15, ++--+, 1)$.
(c) The general structure of the dual tree.}
\label{fig:middle_odd_tree_new}
\end{figure}

\begin{proposition}\label{thm:middle_odd_struct}
Let $M$ be a DBD-matching of size $k$. Then:
\begin{enumerate}
  \item $M$ has exactly $\ell-1$ neighbors (where $\ell = \left\lceil \frac{k}{2} \right\rceil$);
  \item All the neighbors of $M$ are leaves.
\end{enumerate}
  Thus, the connected component that contains $M$ is a star of order $\ell$.
\end{proposition}

\begin{proof}\mbox{}
\begin{enumerate}
\item
Let $M'$ be a (supposed) neighbor of $M$.
Consider the corresponding flippable partition of $M$.
Its members can be of size at most $3$
because inner faces of $M$ have at most three edges.
Since $k$ is odd, there is at least one set of size $3$ in the flippable partition,
which must be a DBD-element $\{d_j, b_j, d_{j+1}\}$ ($1 \leq j \leq \ell-1$).
The parts of $M$ to the left and to the right of this DBD-element
are DB-matchings (if non-empty), and, therefore,
upon the choice of a DBD-element that belongs to a flippable partition,
the construction of a
disjoint compatible matching can be completed in a unique way.
Since $M$, with this flippable partition (shown by square brackets) has the pattern
\[
\underbrace{\mathrm{[DB] \dots [DB]}}_{(j-1)\times\mathrm{DB}} \mathrm{[DBD]}  \underbrace{\mathrm{[BD] \dots [BD]}}_{(\ell-1-j)\times\mathrm{BD}},
\]
the matching $M'$ determined by flipping the $j$th DBD-element has by Observation~\ref{thm:db_flip} the pattern
\[
\underbrace{\mathrm{[BD] \dots [BD]}}_{(j-1)\times\mathrm{BD}}  \mathrm{[B^{2+1}]} \underbrace{\mathrm{[DB] \dots [DB]}}_{(\ell-1-j)\times\mathrm{DB}}.
\]
The position of B-edges of $M'$ matches that of $M$.
Denote this matching $M'$ by $\zsl(k, j, \chi, z)$.

The dual tree of $M'=\zsl(k, j, \chi, z)$
is obtained from that of $M=\zsc(k, \chi, z)$
by erasing the edges $B_0D_1$ and $D_{\ell-1}B_{\ell}$, 
and attaching two additional leaves, one below the path and one above it, to $D_{j}$.
The edge side $D_1B_1$ is labeled by $z$.

Since we have $\ell-1$ ways to choose the DBD-element that belongs to a flippable partition,
$M$ has $\ell-1$ neighbors.

\item
We see inductively that the only flippable partition of a DBDL-matching consists
of $\ell-2$ DB-elements and one $\mathrm{B^{2+1}}$-element.
Therefore, it has only one neighbor, and, thus, it is an L-matching.

\end{enumerate}
\end{proof}

Figure~\ref{fig:middle_odd_ex_new_1}
shows the matching $\zsc(11, ++-, 1)$,
its neighbors $\zsc(11, j, ++-, 1)$, $1\leq j\leq 5$,
and their dual trees.
For the DBDL-matchings, the flippable sets
are marked by a blue box.

\begin{figure}[h]
$$\includegraphics[width=150mm]{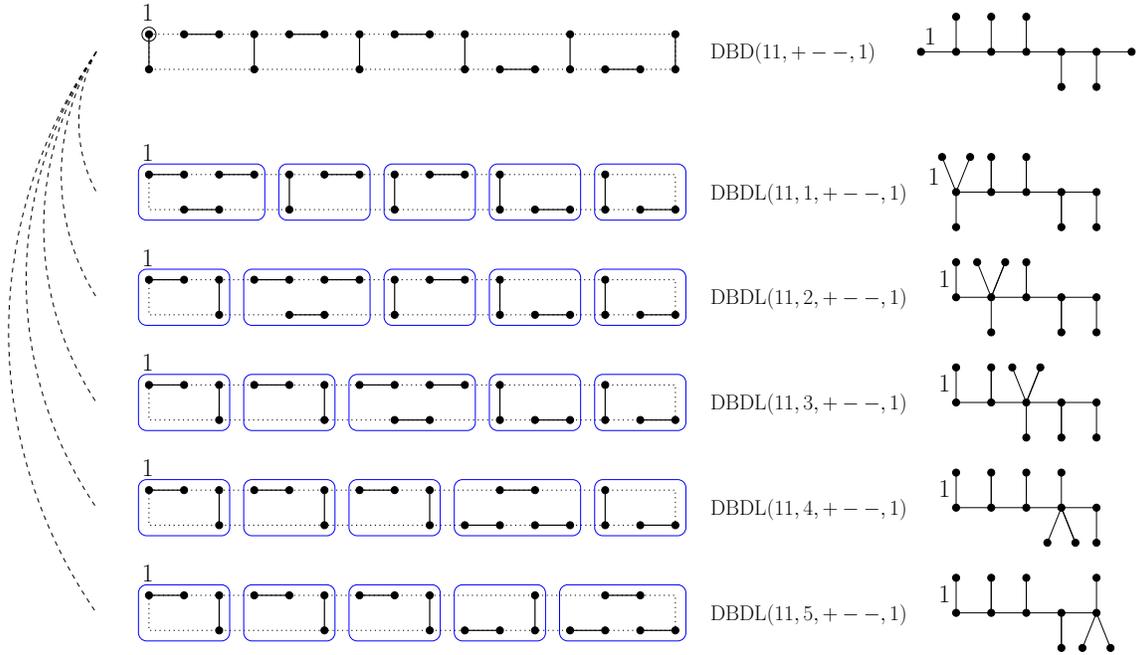}$$
\caption{The matching $\zsc(11, ++-, 1)$, its neighbors, and their dual trees.}
\label{fig:middle_odd_ex_new_1}
\end{figure}


\begin{proposition}\label{thm:middle_odd_enum}
The number of DBD-matchings of size $k$ is $(2\ell-1)\cdot 2^{\ell - 3}$.
\end{proposition}

\begin{proof}
For a DBD-matching of size $k$,
$\chi$ can be chosen in $2^{\ell - 3}$ ways,
and $z$ in $2k=2(2\ell-1)$ ways.
However, as explained above, $\zsc(k, \chi, z) = \zsc(k, \chi', z')$,
and this is the only way to represent a DBD-matching by a standard strip drawings in several ways.
Therefore, each DBD-matching is represented in this way exactly twice.
It follows that there are $(2\ell-1) \cdot 2^{\ell - 3}$ DBD-matchings.
\end{proof}

\begin{corollary}\label{thm:stars_enum}
The number of connected components of $\dcm_k$
that contain DBD- and DBDL-matchings is $(2\ell-1) \cdot 2^{\ell - 3}$.
\end{corollary}

To summarize: In this section we described certain connected components of $\dcm_k$ for odd values of $k$.
The enumerational results fit those from Table~\ref{tab:odd}.
In Section~\ref{sec:big} we will show that these are precisely the medium components of $\dcm_k$ for odd $k$.


\subsection{Medium components for even $k$}\label{sec:middle_even}


Recall the definition of DB-matching from Section~\ref{sec:small_even}.
Refer again to Figure~\ref{fig:small_even_tree_new_2} for
the standard representation of a DB-matching by a strip drawing,
and for the labeling of its edges and faces.
In particular,
the standard drawing of a DB-matching of size $k-2$ has
$\ell-1$ faces $D_1, \dots, D_{\ell-1}$ (from left to right).

\medskip

\noindent\textbf{Definition.}
An \textit{EDB-matching}\footnote{\textit{EDB} stands for ``extended DB-matching''.} of size $k$
is a matching whose (standard) stripe drawing can be
obtained from that of a DB-matching of size $k-2$
by adding two boundary edges to one of the faces $D_j$ ($1 \leq j \leq \ell-1$),
one on $\zbu$ and one on $\zbl$
(or, equivalently, by replacing one of its DB-elements
by an EDB-element of the same direction and position).

\medskip

Thus, a DB-matching of size $k-2$ produces $\ell-1$ EDB-matchings of size $k$.
Specifically, let $\db(k-2, \chi, z)$ be a DB-matching.
For each $j$, $1 \leq j \leq \ell-1$,
we denote by $\edb(k, j, \chi, z)$,
the matching obtained from $\db(k-2, \chi, z)$
by adding two boundary edges, as explained above, to $D_{j}$.
These two boundary edges will be denoted by $e$ and $e'$:
$e$ lies on the same side of $\zbr$ as $b_j$
(in order to distinguish between $b_j$ and $e$ we assume that $e$ is to the left of $b_j$),
and $e'$ on the opposite side.

Equivalently, the dual tree of an EDB-matching of size $k$
can be obtained from the dual tree of a DB-matching of size $k-2$
by attaching a pair of leaves, $E$ and $E'$,
one below and one above the path $B_0 \dots D_{\ell-1}$,
to one of the vertices $D_j$, $1 \leq j \leq \ell-1$
(the edges $D_jE$ and $D_jE'$ correspond, respectively, to $e$ and $e'$).
See Figure~\ref{fig:ebd2} for an example.
\begin{figure}[h]
$$\includegraphics[width=140mm]{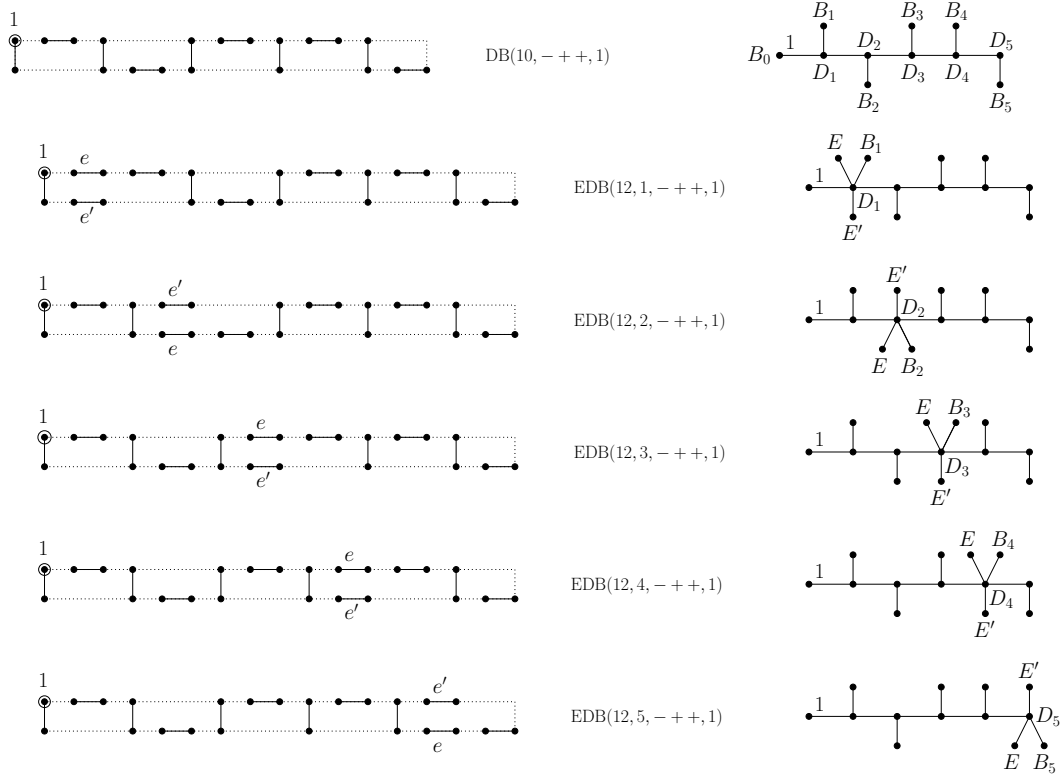}$$
\caption{The five EDB-matchings
$\edb(12, j, -++, 1)$, $j=1,2,3,4,5$,
 produced by $M=\db(10, -++, 1)$.}
\label{fig:ebd2}
\end{figure}

Recall from the proof of Theorem~\ref{thm:small_even_struct}
that the only neighbor of $\db(k-2, \chi, z)$
is $\db(k-2, \chi', z')$, where
$z'=z+(k-2)+\delta(\chi)$.

\begin{proposition}\label{thm:ebd_neighbors}
The EDB-matching $M=\edb(k, j, \chi, z)$ has $j+2$ neighbors, namely:
\begin{itemize}
\item $j$ EDB-matchings, namely, $\edb(k, i, \chi', z')$ for
$\ell-j \leq i \leq \ell-1$
(here $z'=z+k+\delta(\chi)$);
\item and two L-matchings. 
\end{itemize}
\end{proposition}

\begin{proof}
Consider the standard strip drawing of $M=\edb(k, j, \chi, z)$.
Let $M'$ be a (supposed) neighbor of $M$.
The set $P=\{d_j, b_j, e, e'\}$ is an R-directed EDB-element of $M$.
The part of $M$ to the right of $P$
is (if non-empty) a DB-matching consisting of R-directed DB-elements,
and, therefore, they are replaced in $M'$ by L-directed DB-elements with the same position.
The edges of $P$ can belong to the sets from a flippable partition in several ways.
There are several cases to consider.

\begin{itemize}
\item \textbf{Case 1: The quadruple $P=\{d_{j}, b_j, e, e'\}$ belongs to the flippable partition.}
$P$, the R-directed EDB-element of $M$, is replaced in $M'$ by an L-directed EDB-element with the same position.
If there are edges to the left of $P$, they form a DB-matching consisting of R-directed DB-elements.
Thus, in $M'$ they are replaced in $M'$ by L-directed elements with the same position.
Since $M$ with its flippable partition has the form
\[
\underbrace{\mathrm{[DB] \dots [DB]}}_{j\times\mathrm{DB}}  \mathrm{[DB^{2+1}]} \underbrace{\mathrm{[DB] \dots [DB]}}_{(\ell-1-j)\times\mathrm{DB}},
\]
we obtain that $M'$ has the form
\[
\underbrace{\mathrm{[BD] \dots [BD]}}_{j\times\mathrm{BD}} \mathrm{[B^{2+1}D]}  \underbrace{\mathrm{[BD] \dots [BD]}}_{(\ell-1-j)\times\mathrm{BD}},
\]
that is, $M'$ is also an EDB-matching (drawn upside down),
namely,
$M'=\edb(k, \ell-j, \chi', z')$.
See Figure~\ref{fig:ebd_neighbors_flip4_new} for an example.
\begin{figure}[h]
$$\includegraphics[width=150mm]{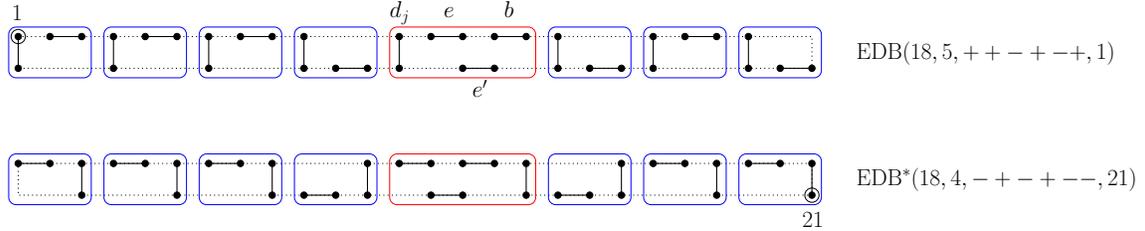}$$
\caption{$\edb(18, 5, ++-+-+, 1)$
and its neighbor $\edb(18, 4, -+-+--, 21)$
determined by flipping a quadruple
(Proposition~\ref{thm:ebd_neighbors}, case $1$).}
\label{fig:ebd_neighbors_flip4_new}
\end{figure}

\item \textbf{Case 2: The triple $\{b_j, e, e'\}$ belongs to the flippable partition.}
This triple is a $\mathrm{B^{2+1}}$-element.
Upon flipping it, we obtain in $M'$ a DBD-element with the same position.
The part of $M$ to the left of this triple, is
(if non-empty)
a DBD-matching of size $2j-1$.
Therefore, it follows from the proof of Proposition~\ref{thm:middle_odd_struct},
that $M'$ is determined by flipping another flippable DBD-element --
$\{d_i, b_i, d_{i+1}\}$ for some $ 1 \leq i \leq j-1$.
Since $M$ has the form
\[
\underbrace{\mathrm{[DB] \dots [DB]}}_{(i-1)\times\mathrm{DB}}
\mathrm{[DBD]}
\underbrace{\mathrm{[BD] \dots [BD]}}_{(j-i)\times\mathrm{BD}}
\mathrm{[B^{2+1}]}
\underbrace{\mathrm{[DB] \dots [DB]}}_{(\ell-1-j)\times\mathrm{DB}},
\]
we obtain that $M'$ has the form
\[
\underbrace{\mathrm{[BD] \dots [BD]}}_{(i-1)\times\mathrm{BD}}
\mathrm{[B^{2+1}]}
\underbrace{\mathrm{[DB] \dots [DB]}}_{(j-i)\times\mathrm{DB}}
\mathrm{[DBD]}
\underbrace{\mathrm{[BD] \dots [BD]}}_{(\ell-1-j)\times\mathrm{BD}},
\]
which can be rewritten as
\[
\underbrace{\mathrm{BD \dots BD}}_{(i-1)\times\mathrm{BD}}
\mathrm{B^{2+1}} \mathrm{D}
\underbrace{\mathrm{BD \dots BD}}_{(\ell-i)\times\mathrm{BD}},
\]
which means that $M'$ is also an EDB-matching (drawn upside down),
namely -- since the position of the flipped elements didn't change, --
$M'=\edb(k, \ell-i , \chi', z')$.

Since the flippable DBD-element can be chosen in $j-1$ ways, we obtain in this case $j-1$ neighbors of $M$.
See Figure~\ref{fig:ebd_neighbors_flip3_new} for an example
(the flipped triples are indicated by red boxes around the matchings adjacent to $M$).
\begin{figure}[h]
$$\includegraphics[width=150mm]{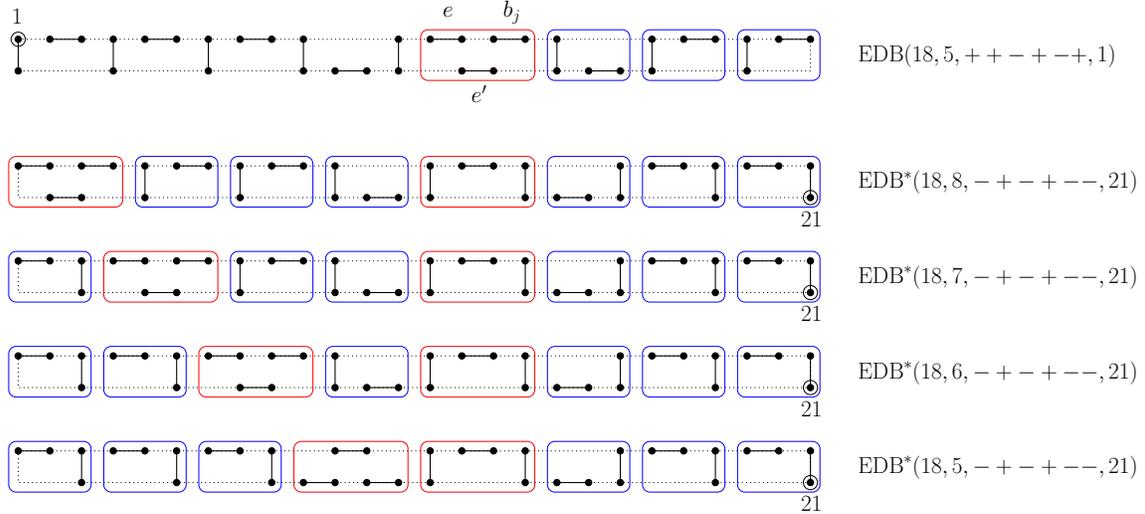}$$
\caption{$\edb(18, 5, ++-+-+, 1)$
and its neighbors $\edb(18, j, -+-+--, 21)$, $5 \leq j \leq 8$,
determined by flipping two triples
(Proposition~\ref{thm:ebd_neighbors}, case $2$).}
\label{fig:ebd_neighbors_flip3_new}
\end{figure}

\item \textbf{Case 3a: Two pairs, $\{b_j, e\}$ and $\{d_{j}, e'\}$, belong to the flippable partition.}

$M \setminus \{b_j, e\}$
is the DB-matching obtained from $\db(k-2, \chi, z)$ by changing the position of its $j$th DB-element.
Thus, the neighbor of $M \setminus \{b_j, e\}$ is
the DB-matching obtained from $\db(k-2, \chi', z')$
by changing the position of its $(\ell-j)$th DB-element.
The antiblock $\{b_j, e\}$ of $M$ is replaced in $M'$ by the block
inserted in the $(\ell-j-1)$st face of $\db(k-2, \chi', z')$
on the side corresponding to the position of its $(\ell-j)$st DB-element
(if the B-edge of the $(\ell-j-1)$st face is also on this side, then this block is
closer to $(\ell-j)$th face -- to the right in the standard drawing of $\db(k-2, \chi', z')$, but to the left in our upside down drawing).
%

We denote this $M'$ by $\edbl_1(k, j, \chi, z)$.
Since is is obtained from a DB-matching by inserting a block, it is an L-matching.
See Figure~\ref{fig:ebd_neighbors_flip2_2}(a) for an example.
It also shows the general form of corresponding dual trees.
The dotted line surrounding a leaf and a $2$-branch
indicates that these branches are on the different sides of the path.

\item \textbf{Case 3b: Two pairs, $\{b_j, e'\}$ and $\{d_{j}, e\}$, belong to the flippable partition.}
$M \setminus \{b_j, e'\}$
is the DB-matching $\db(k-2, \chi, z)$.
Its neighbor is $\db(k-2, \chi', z')$.
The flippable pair $\{b_j, e'\}$ is replaced in $M'$ by a two D-edges.
Thus, $M'$ can be obtained from $\db(k-2, \chi', z')$
by replacing its $(\ell-j)$th D-edge by three D-edges.

We denote this $M'$ by $\edbl_2(k, j, \chi, z)$.
It can be obtained by inserting a block (DD) into a DB-matching consisting of $\ell-j-1$ DB-elements (its right side),
and then inserting $j$ blocks (its left side).
Therefore it is an L-matching.
See Figure~\ref{fig:ebd_neighbors_flip2_2}(b) for an example.
\begin{figure}[h]
$$\includegraphics[width=150mm]{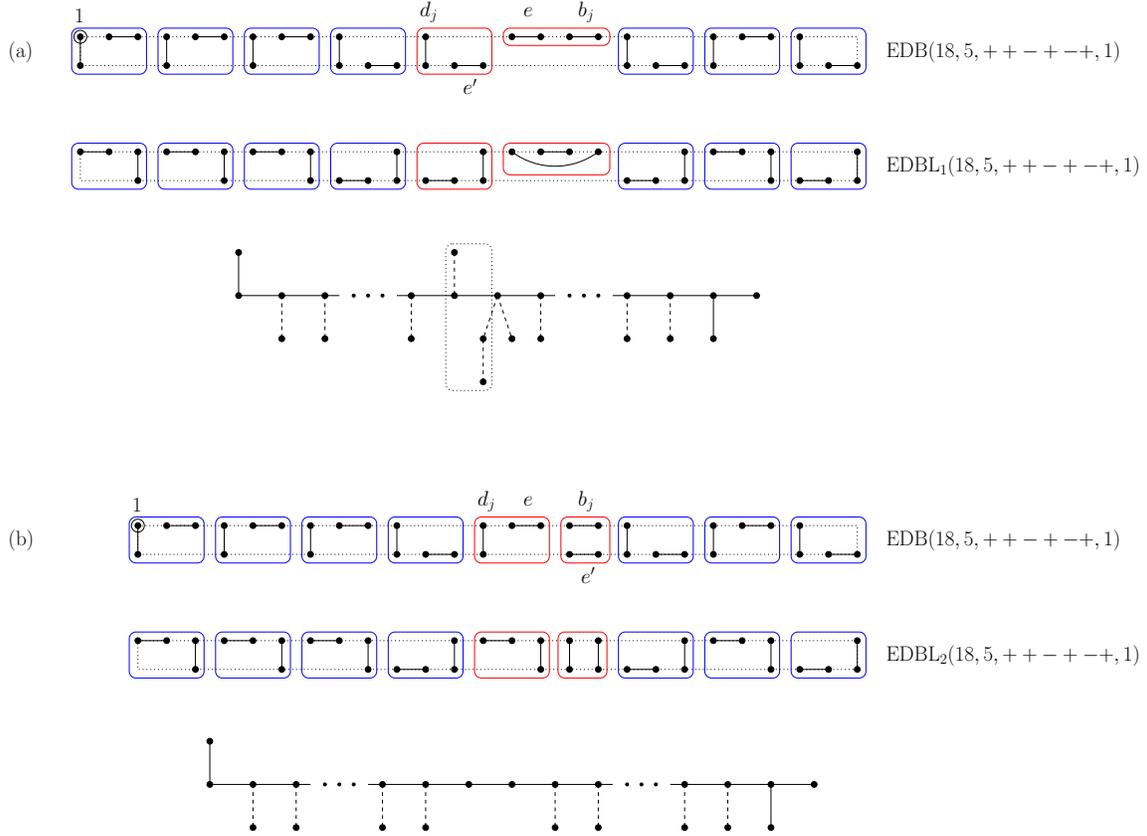}$$
\caption{$\edb(18, 5, ++-+-+, 1)$
and its neighbors
determined by flipping two pairs in $D_j$
(Proposition~\ref{thm:ebd_neighbors}, cases $3$a and $3$b).}
\label{fig:ebd_neighbors_flip2_2}
\end{figure}
\end{itemize}
\end{proof}

\noindent\textit{Remark.}
We showed that EDBL-matchings can be obtained from DB-matchings by inserting certain elements.
In some cases (listed below), when these elements are inserted close to the either of the ends,
the obtained EDBL-matchings, and, correspondingly, their dual trees, have some special elements
that do not present in the ``regular'' cases.
For $j=1$,
the dual graph of $\mathrm{EDBL}_1$ has a vertex of degree $4$ to which two $2$-branches are attached, and
the dual graph of $\mathrm{EDBL}_2$ a $4$-branch.
For $j=\ell-1$,
the dual graph of $\mathrm{EDBL}_1$ and that of $\mathrm{EDBL}_2$ have $3$-branches.
For $j=\ell-2$,
the dual graph of $\mathrm{EDBL}_1$ has a vertex of degree $4$ to which two leaves and one $4$-branch are attached.
See Figure~\ref{fig:ebd_neighbors_spec_new} for an example
and the general structure of dual trees in such cases.
\begin{figure}[h]
$$\includegraphics[width=165mm]{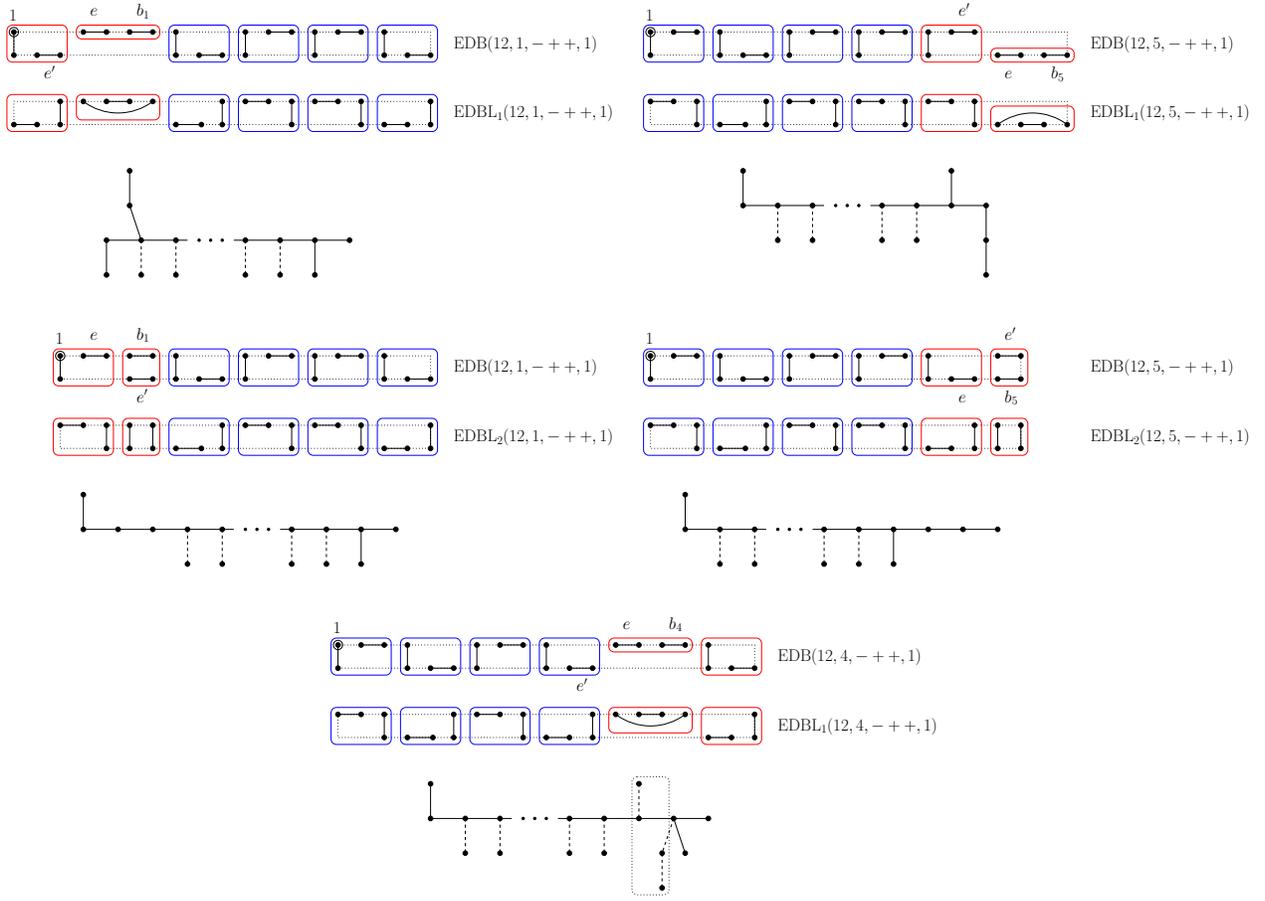}$$
\caption{EDBL-matchings with special structure (Illustration to remark to Proposition~\ref{thm:ebd_neighbors}).}
\label{fig:ebd_neighbors_spec_new}
\end{figure}

\medskip

Since the neighbors of an EDB-matching $M=\edb(k, j, \chi, z)$
are only EDB-matchings with parameters $\chi'$ and $z'$, and two L-matchings,
the structure of the connected component of $\dcm_k$ that contains $M$
follows from Proposition~\ref{thm:ebd_neighbors}.

\begin{corollary}\label{thm:middle_even_struct}
The connected component of $\dcm_k$
that contains $\edb(k, j, \chi, z)$ has the following structure:
\begin{itemize}
\item There is a path $P$ of length $k-3$:
\[\edb(k, 1, \chi, z) -
\edb(k, \ell-1, \chi', z') -
\edb(k, 2, \chi, z) -
\edb(k, \ell-2, \chi', z') - \ldots \ \ \ \ \ \ \ \ \]
\[\ \ \ \ \ \ \ \ \ldots -
\edb(k, \ell-2, \chi, z) -
\edb(k, 2, \chi', z') -
\edb(k, \ell-1, \chi, z) -
\edb(k, 1, \chi', z');
\]
\item There are additional edges between the matchings that belong to $P$, as follows:
\[\edb(k, j_1, \chi, z) -
\edb(k, j_2, \chi', z')\]
for all $j_1, j_2$ ($1 \leq \{j_1, j_2\}\leq \ell-1$) such that $j_1+j_2 \geq \ell+2$;\\
(Equivalently: if we denote the matchings from the path $P$,
according to the order in which they appear on $P$,
by $M_1, M_2, \dots, M_{k-2}$,
then these additional edges are all the edges of the form $M_a M_b$,
where $a$ is even, $b$ is odd, and $a \leq b-3$.)
\item Each member of $P$ is also adjacent to two leaves.
\end{itemize}
In particular,
all such components are isomorphic,
and their size is $3(k-2)$.
\end{corollary}

Figure~\ref{fig:middle_even_component_2} shows such a component for $k=12$.
The labels $(12, j, \chi/ \chi', z/z')$
(with ``EDB'' being omitted)
refer to the vertices of the path $P$ that appear directly above them.

\begin{figure}[h]
$$\includegraphics[width=150mm]{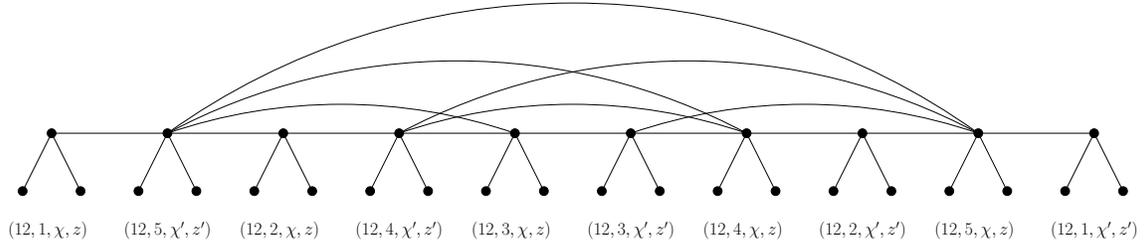}$$
\caption{The structure of the connected component of
$\dcm_{12}$ that contains an EDB-matching.}
\label{fig:middle_even_component_2}
\end{figure}

\begin{proposition}\label{thm:middle_even_enum}
The number of components of $\dcm_k$ that contain EDB-matchings is $\ell\cdot 2^{\ell - 2}$.
\end{proposition}

\begin{proof}
By Proposition~\ref{thm:bd_enum}, the number of DB-matchings of size $k-2$ is $(\ell-1) \cdot 2^{\ell-1}$.
Therefore, there are $2^{\ell-4}$ pairs of unlabeled DB-matchings of size $k-2$.
Each such pair produces one connected component that contains unlabeled EDB-matchings of size $k$.
$z$ can be chosen in $2k = 4\ell$ ways.
Thus, the number of such components is $ \ell\cdot 2^{\ell - 2}$.
\end{proof}

To summarize: In this section we described certain connected components of $\dcm_k$ for even values of $k$.
The enumerational results fit those from Table~\ref{tab:even}.
In Section~\ref{sec:big} we will show that these are precisely the medium components of $\dcm_k$ for even $k$.

\section{Big components}\label{sec:big}
\subsection{The survey of the proof}\label{sec:big_plan}

In Section~\ref{sec:small} we defined I- and DB-matchings and proved that they are precisely those matchings that form small components.
In Section~\ref{sec:middle} we defined DBD-, DBDL-, EDB- or EDBL-matchings and described their connected components.
In order to complete the proof, we need to show that all other matchings form one (``big'') connected component.
We start with some definitions.

\medskip

\noindent\textbf{Definitions.}\mbox{}
\begin{enumerate}
\item The \textit{ring component} of $\dcm_k$ is the connected component that contains the rings.
\item A \textit{special} matching is either an $\mathrm{I}$-,  $\mathrm{DB}$-, $\mathrm{DBD}$-, $\mathrm{DBDL}$-, $\mathrm{EDB}$- or $\mathrm{EDBL}$-matching.
\item A \textit{regular} matching is a matching which is not special.
\end{enumerate}

\medskip

Observe that for $k \geq 5$ the rings are regular matchings.

Theorem~\ref{thm:main_theorem} follows from
the results obtained above and the following theorem.

\begin{theorem}\label{thm:ring_component}
For $k \geq 9$, every regular matching $M$ belongs to the ring component.
\end{theorem}

\begin{proof}
For $k=9$ and $10$, the statement was verified by a computer program.
For $k \geq 11$, the proof is by induction.

By Proposition~\ref{thm:exist_separated}, $M$ has at least one separated pair $K$.
Let $L = M - K$.
Now we have two cases depending on whether $L$ is special or regular.

\smallskip

\noindent\textbf{Case \boldmath 1: $L$ is regular.}
By induction, $L$ belongs to the ring component in $\dcm_{k-2}$.
We perform the sequence of operations that converts $L$ into a ring, while $K$ oscillates
(that is, on the points of $K$, on each step a block is replaced by an antiblock, or vice versa).
In this way we obtain a matching of the form $R+K'$ where $R$ is a ring of size $k-2$ and $K'$ is a separated pair.
We can also assume that $K'$ is an antiblock
(otherwise, if $K'$ is a block, we flip $K'$ and $R$:
$K'$ is then replaced by an antiblock, and $R$ by the second ring).
If the antiblock $K'$ is inserted in a skip of $R$,
then the whole obtained matching is a ring of size $k$, and we are done.
Otherwise, the antiblock $K'$ is inserted between two connected points of $R$.
In such a case we use the following proposition that will be proven in
Section~\ref{sec:non_bip}.

\begin{proposition}\label{thm:non_bip}
For $k \geq 8$, the ring component of $\dcm_k$ is not bipartite.
\end{proposition}

Thus, it is possible to convert the ring $R$
into the second ring by an \textit{even} number of operations.
We perform these operations, while $K'$ oscillates.
After this sequence of operations, we still have the antiblock $K'$,
but the ring $R$ is replaced by the second ring $R'$, and now the whole matching is a ring of size $k$.
Figure~\ref{fig:ring_antiblock_n_1} illustrates the last step for odd $k$.
\begin{figure}[h]
$$\includegraphics[width=70mm]{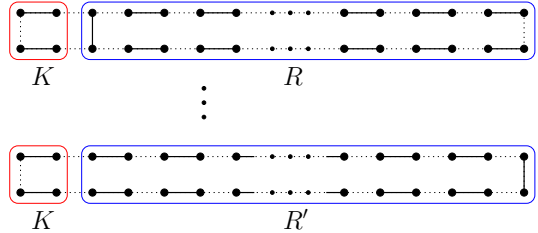}$$
\caption{Illustration to the proof of Theorem~\ref{thm:ring_component} when $L$ is regular.}
\label{fig:ring_antiblock_n_1}
\end{figure}

This completes the proof of Case 1.

\smallskip

\noindent\textbf{Case \boldmath 2: $L$ is special.}
In this case we use the following proposition that will be proven in Section~\ref{sec:big_proof}.

\begin{restatable}{proposition}{bigproof}\label{thm:big_proof}
Let $M$ be a regular matching of size $k$ ($k \geq 10$)
that has a decomposition $M=L+K$ where $K$ is a separated pair and $L$ is a special matching.
Then $M$ has another decomposition $N+P$,
where $P$ is a separated pair and $N$ is a \textbf{regular} matching,
or $M$ is connected (by a path) to a matching that has such a decomposition.
\end{restatable}

Thus, $M$ has a decomposition as in Case~1,
or it is connected by a path to a matching that has such a decomposition.
In both cases it means that $M$ belongs to the ring component.
This completes the proof.
\end{proof}

It remains to prove Propositions~\ref{thm:non_bip}
and~\ref{thm:big_proof}.

\subsection{The ring component is not bipartite for $k \geq 8$
(proof of Proposition~\ref{thm:non_bip}).}\label{sec:non_bip}

We prove Proposition~\ref{thm:non_bip} by constructing a path of odd length from a ring to itself.
In figures, we mark the matchings alternatingly by white and black squares,
starting with a ring marked by white.
We finish when we obtain the same ring marked by black.

\smallskip

First we prove the proposition for even values of $k$.
For $k=8$, it is verified directly, see Figure~\ref{fig:bip_8_n}
(in this and the following figures, we use ``vertical'' strip drawings in order to save the space).
\begin{figure}[h]
$$\includegraphics[width=120mm]{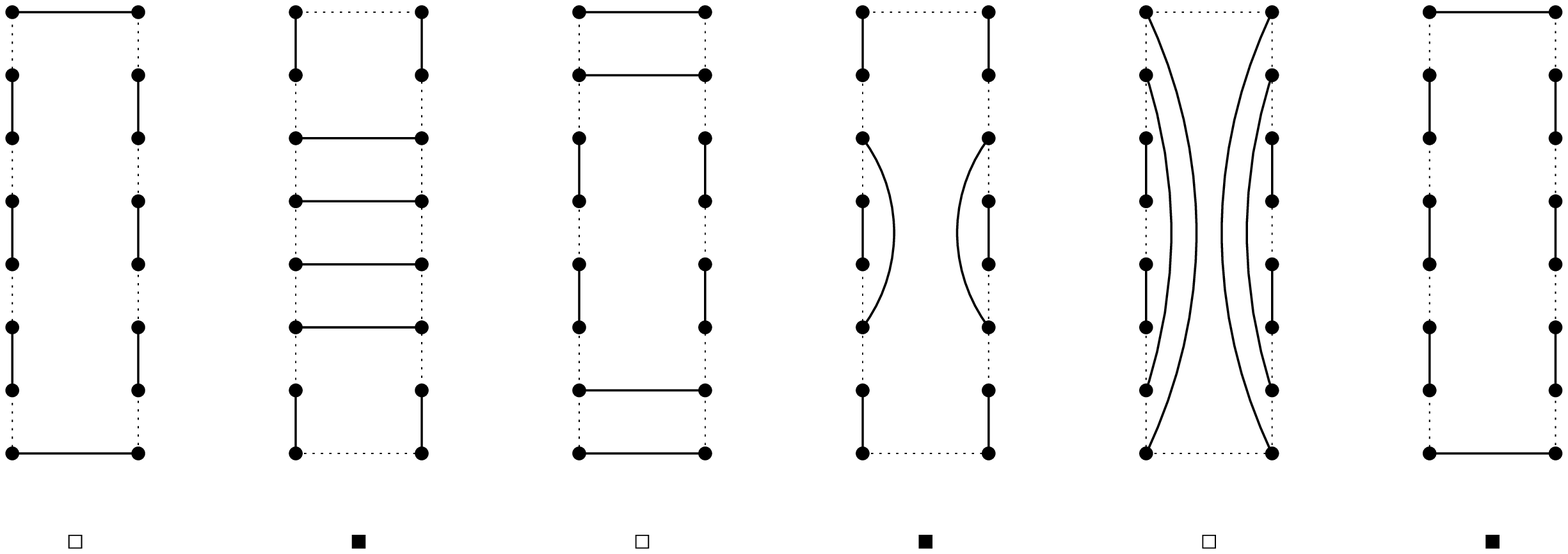}$$
\caption{Proof of Proposition~\ref{thm:non_bip} for $k=8$.}
\label{fig:bip_8_n}
\end{figure}

\begin{figure}[h]
$$\includegraphics[width=110mm]{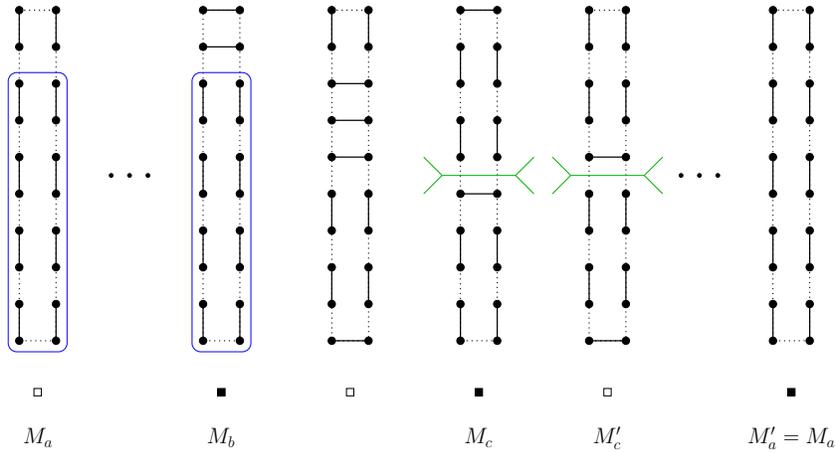}$$
\caption{Proof of Proposition~\ref{thm:non_bip} for $k=10$.}
\label{fig:bip_10_n}
\end{figure}

For $k=10$
refer to Figure~\ref{fig:bip_10_n}.
We start with  a ring $M_a$ represented by a strip drawing.
$M_b$ is obtained from $M_a$ by applying the operations as in Figure~\ref{fig:bip_8_n}
on the flippable set of size $8$ marked by a blue box.
Since the number of these operations is odd, the block outside this flippable set is replaced by an antiblock.
After the next two steps we reach a drawing $M_c$.
For each drawing $M_i$ on the path from $M_a$ to $M_c$, denote by $M'_i$ the reflection of $M_i$ with respect to the green line (which halves the points).
Notice that $M'_c$ is adjacent to $M_c$.
Therefore, we can obtain the path $M_a \dots M_b  M_c  M'_c  M'_b \dots M'_a $.
This path has odd length, and $M'_a=M_a$.
Thus, we have found a path of odd length from a ring to itself.

For even $k \geq 12$ we prove the statement by induction, assuming it holds for $k-4$ and for $k-2$.
Refer to Figure~\ref{fig:bip_even_n}.
We start from a ring $M_a$.
$M_b$ is obtained from $M_a$ by applying the odd number of operations
which transfer the ring of size $k-2$ to itself,
on the flippable set marked by blue.
$M'_b$ is obtained from $M_b$ by applying the odd number of operations
which transfer the ring of size $k-4$ to itself,
on the flippable set marked by red boxes.
Notice that $M'_b$ is the reflection of $M_b$ with respect to the green line (which halves the points).
Therefore, we can obtain the path $M_a \dots M_b  \dots M'_b \dots M'_a $, where
$M'_i$ is the reflection of $M_i$ with respect to the green line.
This path has odd length, and $M'_a=M_a$.
Thus, we have found a path of odd length from $M_a$ to itself.
\begin{figure}[h]
$$\includegraphics[width=80mm]{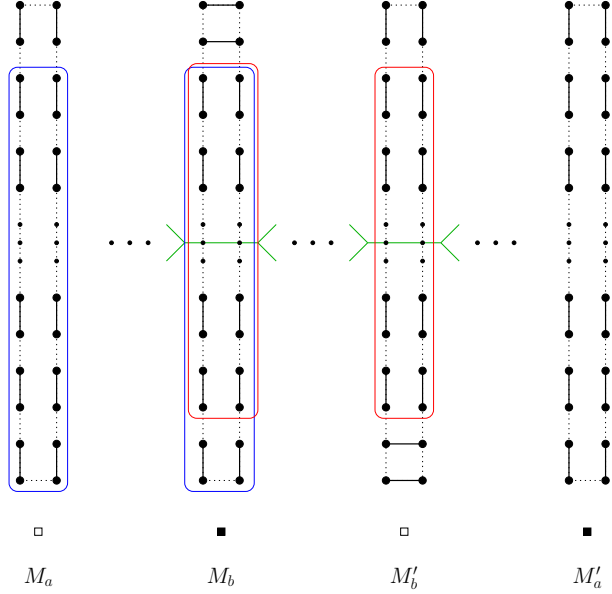}$$
\caption{Proof of Proposition~\ref{thm:non_bip} for even $k\geq 12$.}
\label{fig:bip_even_n}
\end{figure}

\smallskip

Now we prove the proposition for odd values of $k$.
For $k=9$, it is verified directly.
Refer to Figure~\ref{fig:bip_9_n}.
We start from a ring $M_a$,
and after four steps we reach a matching $M_c$
which is symmetric with respect to the green line.
Therefore we can construct a path of even size
 $M_a \dots M_b  M_c M'_b \dots M'_a $, where
 $M'_i$ the reflection of $M_i$ with respect to the green line.
 $M'_a$ is the second ring, which is disjoint compatible to $M_a$,
 and, thus we have a path of odd length from $M_a$ to itself.
\begin{figure}[h]
$$\includegraphics[width=120mm]{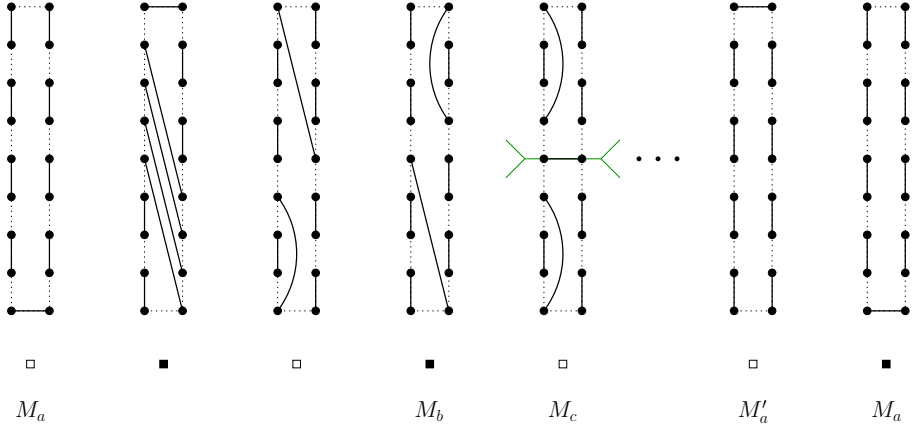}$$
\caption{Proof of Proposition~\ref{thm:non_bip} for $k=9$.}
\label{fig:bip_9_n}
\end{figure}

For odd $k\geq 11$,
we prove the statement using the even case proven above.
Refer to Figure~\ref{fig:bip_odd_n}.
We start from a ring $M_a$.
$M_b$ is obtained from $M_a$ by applying an odd number of operations
on the flippable set of size $k-3$ marked by blue,
while the remaining flippable triple oscillates.
After two more steps we reach a matching $M_d$,
which is symmetric with respect to the green line.
Therefore we can construct a path of even size
 $M_a \dots M_b  M_c M_d M'_c M'_b \dots M'_a $, where
 $M'_i$ is the reflection of $M_i$ with respect to the green line.
 $M'_a$ is the second ring which is disjoint compatible to $M_a$.
 Thus we have a path of odd length from $M_a$ to itself.
 \hfill\qed
\begin{figure}[h]
$$\includegraphics[width=110mm]{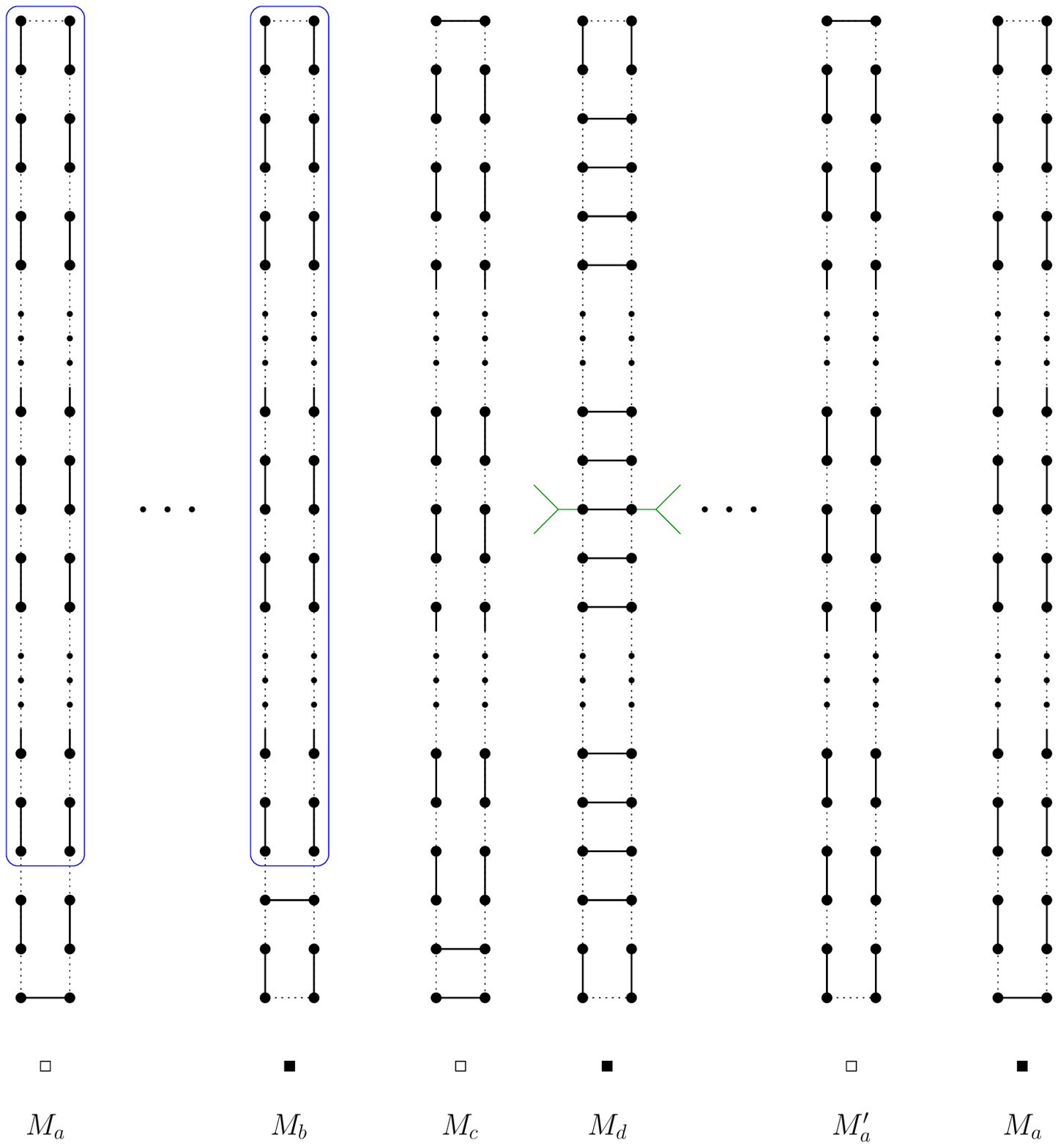}$$
\caption{Proof of Proposition~\ref{thm:non_bip} for odd $k\geq 11$.}
\label{fig:bip_odd_n}
\end{figure}

\medskip

\noindent\textit{Remark.} We have verified by direct inspection and a computer program
that for $2 \leq k \leq 7$, the ring component of $\dcm_k$ \textit{is} bipartite.

\subsection{Proof of Proposition~\ref{thm:big_proof}}\label{sec:big_proof}

We restate the claim to be proven in this section.

\bigproof*

\noindent\textit{Overview of the proof.}
In the proof to be presented,
the possible structure of $L$ plays the central role,
and we need to refer to the definitions
and standard notation of some kinds of special matchings.
Therefore we replace $k$ by $k-2$,
and assume from now on that
$L$ is a matching of size $k$ and
$M$ is a matching of size $k+2$,
where $k \geq 8$.

Since the special matchings have different structure for odd and even values of $k$, the proofs for these cases are separate.
It is more convenient to follow the proofs if we use dual graphs.
In order to simplify the exposition, the elements of the dual graphs
that correspond to blocks and antiblocks --
$2$-branches and V-shapes --
will be occasionally referred to just as blocks and antiblocks.

The idea of the proof is similar to that of the $[\Rightarrow]$-part
in the proof of Theorem~\ref{thm:small_even_struct}.
It is given that $L$ is a special matching.
For some kinds of special matchings we shall proceed as follows.
Depending on the point where $K$ is inserted into $L$
(or, in terms of dual trees, $D(K)$ is attached to $D(L)$),
we shall choose $P$ and show that for this choice
the matching $N=M-P$ does not fit any of the structures
of special matchings (of appropriate parity).
Therefore, $N$ must be regular, and,
thus $M$ \textit{has} a desired decomposition.
For other kinds of special matchings we shall use the structure
of components that contain special matchings
in order to show that $M$ \textit{is connected} 
(by a path)
to a matching that has a
desired decomposition.

\subsubsection{Proof of Proposition~\ref{thm:big_proof} for odd $k$.}\label{sec:big_proof_odd}

First, we recall all possible structures of dual trees of DBD- and DBDL-matchings,
and the standard notation for DBD-matchings.
The dual trees of $\mathrm{DBDL}$-matchings have two possible structures
referred to as $\mathrm{DBDL1}$ and $\mathrm{DBDL2}$,
see Figure~\ref{fig:middle_odd_all_trees_1}.
Moreover, we recall that I-matchings
never have antiblocks (Proposition~\ref{thm:isolated_no_antiblock}),
and that for $k \geq 5$ they have at least two disjoint blocks
(Proposition~\ref{thm:small_odd_two_blocks}).

\begin{figure}[h]
$$\includegraphics[width=90mm]{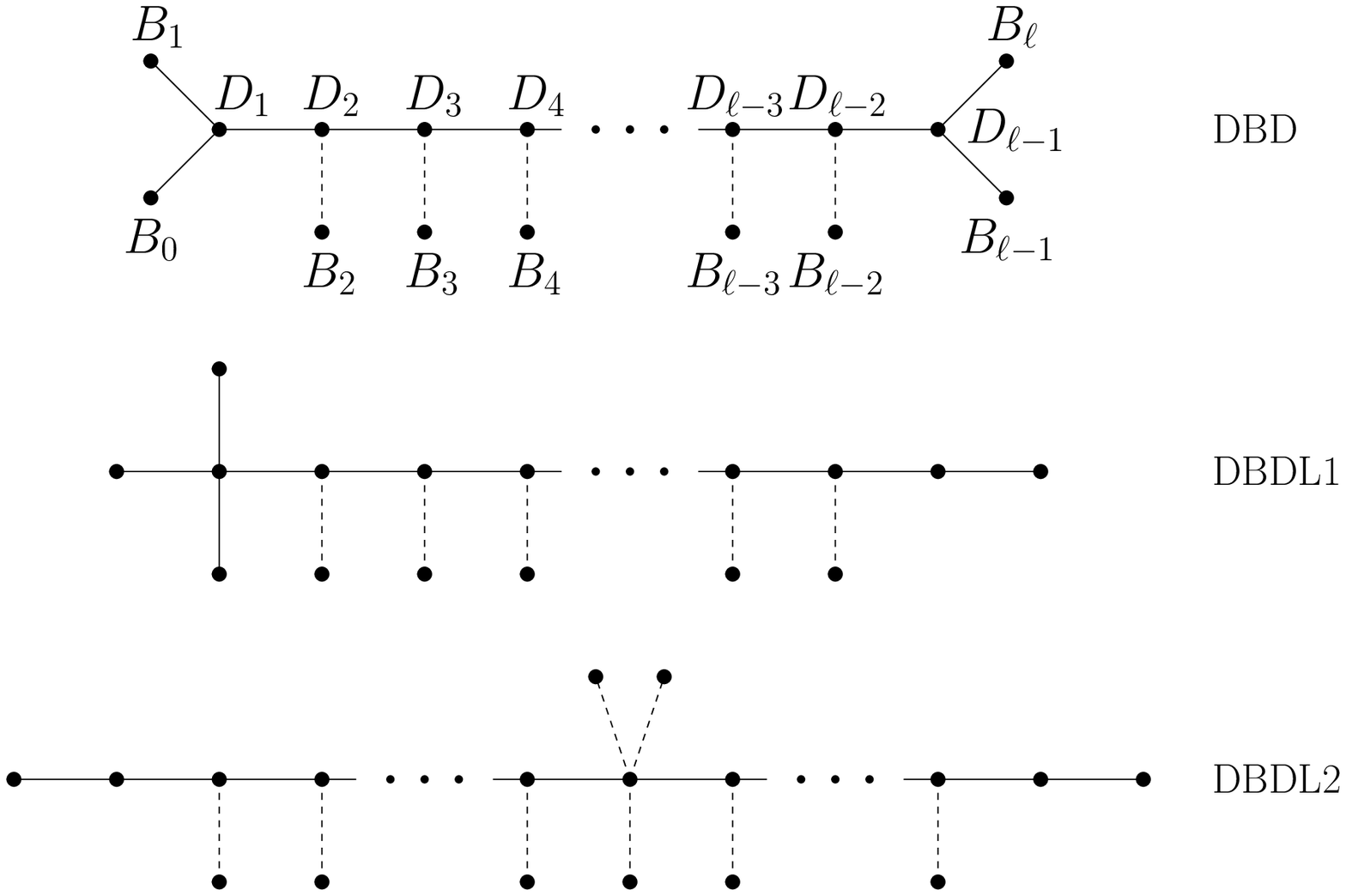}$$
\caption{Dual trees of odd size special matchings from medium components.}
\label{fig:middle_odd_all_trees_1}
\end{figure}

\noindent\textbf{\boldmath Case 1. $L$ is a DBD-matching, $K$ is a block.}
Refer to the first graph in Figure~\ref{fig:middle_odd_all_trees_1} as to the dual tree of $L$.
Due to the symmetry of DBD-matchings, we can assume that $D(K)$ is attached to $D(L)$
in the point $B_i$ or $D_i$ where $i \leq \left\lceil\frac{\ell-1}{2}\right\rceil$.
Let $P$ be the antiblock $B_{\ell}D_{\ell-1}B_{\ell-1}$, and let $N=M-P$.
Then $N$ is a regular matching.
Indeed, $N$ has an antiblock ($D_{\ell-1}D_{\ell-2}B_{\ell-2}$), and thus it cannot be an I-matching.
If $D(K)$ is attached in $B_i$, then $D(N)$ has a $3$-branch, which never happens for DBD- and DBDL-matchings.
If $D(K)$ is attached in $D_i$, then $D(N)$ has a vertex of degree $4$
to which at most two leaves are attached, which never happens for DBD- and DBDL-matchings.

\smallskip

\noindent\textbf{\boldmath Case 2. $L$ is a DBD-matching, $K$ is an antiblock.}
Again we assume that $D(K)$ is attached to $D(L)$
in the point $B_i$ or $D_i$ where $i \leq \left\lceil\frac{\ell-1}{2}\right\rceil$.
Denote by $P$ the antiblock $B_{\ell}D_{\ell-1}B_{\ell-1}$, and let $N=M-P$.
Then $N$ is a regular matching. Indeed, $N$ cannot be an I-matching because it has at least one antiblock.
If $D(K)$ is attached in $B_0$ or $B_1$, then $M$ is special (DBD), while it is assumed to be regular.
If $D(K)$ is attached in $B_i$, $i \geq 2$, then $D(N)$ has three disjoint antiblocks
($K$, $B_0D_1B_1$ and $D_{\ell-1}D_{\ell-2}B_{\ell-2}$),
which never happens for DBD- and DBDL-matchings.
If $D(K)$ is attached in $D_i$, then $D(N)$ has a vertex of degree $5$ and has no blocks, which never happens for DBD- and DBDL-matchings.

\smallskip

\noindent\textbf{\boldmath Case 3. $L$ is a DBDL-matching, $K$ is a separated pair.}
Such a matching $M$ is adjacent to a matching
$M'=L'+K'$
where $L'$ is the DBD-matching adjacent to $L$,
and $K'$ is the flip of $K$.
For $M'$ the statement holds by Cases~1 and~2.
Therefore, it also holds for $M$.

\smallskip

\noindent\textbf{\boldmath Case 4. $L$ is an I-matching, $K$ is a block.}
In such a case $M$ is also an I-matching (by Theorem~\ref{thm:small_even_struct}), and, thus, it cannot be regular.
So, this case is impossible.

\smallskip

\noindent\textbf{\boldmath Case 5. $L$ is an I-matching, $K$ is an antiblock.}
$L$ has at least two disjoint blocks.
Therefore, $M$ has at least one block $K'$.
Clearly, $K'$ is disjoint from $K$.
Denote $L' = M - K'$.
$L'$ cannot be an I-matching because it has an antiblock ($K$).
If $L'$ is a DBD- or a DBDL-matching, we return to Case~1 or~3
(with $L'$ and $K'$ in the role of $L$ and $K$).
If $L'$ is regular, we are done. \hfill\qed

\subsubsection{Proof of Proposition~\ref{thm:big_proof} for even $k$.}\label{sec:big_proof_even}

We recall all possible structures of the dual trees of DB-, EDB- and EDBL-matchings.
As we saw in Section~\ref{sec:middle_even},
the dual tree of any EDB-matching has one of three possible structures,
and
the dual tree of any EDBL-matching has one of six possible structures;
this structures will be referred to as in Figure~\ref{fig:special_even_all_trees_2}
($\mathrm{EDB1}$, $\mathrm{EDB2}$, etc.).
For dual trees of $\mathrm{DB}$-matchings and of $\mathrm{EDB1}$-matchings
(that is, the $\mathrm{EDB}$-matchings in which the edges $e$ and $e'$
belong to the face $D_j$ where $2 \leq j \leq \ell-2$),
we also recall the standard notation of vertices.

\begin{figure}[h]
$$\includegraphics[width=165mm]{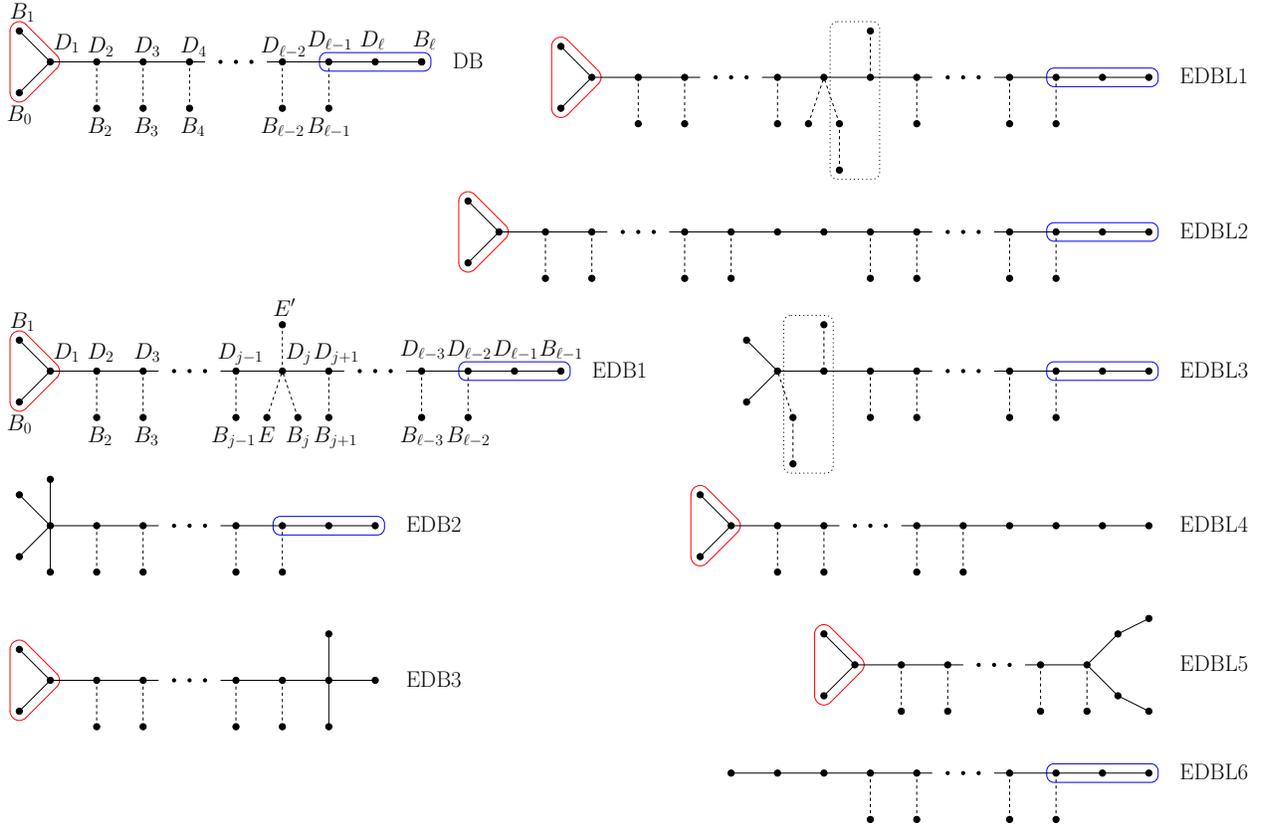}$$
\caption{Dual trees of small and medium special matchings.}
\label{fig:special_even_all_trees_2}
\end{figure}

\noindent\textbf{\boldmath Case 1. $L$ is a $\mathrm{DB}$-matching.}
Refer to the labeling of $D(L)$ as in Figure~\ref{fig:special_even_all_trees_2}.
$D(K)$ is attached to $D(L)$ in some point $B_i$ or $D_i$.
If $i \geq \left \lceil \frac{\ell}{2} \right \rceil $,
let $P$ be the antiblock $B_0 D_1 B_1$.
If $i < \left \lceil \frac{\ell}{2} \right \rceil $,
let $P$ be the block $D_{\ell-1} D_{\ell} B_{\ell}$.
Denote $N=M-P$. We claim that $N$ is regular.

In the case $i \geq \left \lceil \frac{\ell}{2} \right \rceil $,
in the left side of $D(N)$ we have an antiblock $B_{2}D_{2}D_{1}$,
and $D_2$ has degree $3$.
Therefore, if $N$ is special, it can be only a antiblock $Q$ that appears
in the left side of
$\mathrm{DB}$,
$\mathrm{EDB1}$,
$\mathrm{EDB3}$,
$\mathrm{EDBL1}$,
$\mathrm{EDBL2}$,
$\mathrm{EDBL4}$
or
$\mathrm{EDBL5}$
(it is marked by a red frame in Figure~\ref{fig:special_even_all_trees_2}).
However, in such a case, upon restoring $D(P)$
(attaching it to one of the leaves of $Q$)
we obtain a matching that fits the same structure, and therefore, is also special.
This is a contradiction since $M=N+P$ is a regular matching.

In the case $i < \left \lceil \frac{\ell}{2} \right \rceil $
the reasoning is similar:
in the right side of $D(N)$ we have a block $D_{\ell-2} D_{\ell-1} B_{\ell-1}$.
and $D_{\ell-2}$ has degree $3$.
Therefore, if $N$ is special, it can be only a block $R$ that appears
in the right side of
$\mathrm{DB}$,
$\mathrm{EDB1}$,
$\mathrm{EDB2}$,
$\mathrm{EDBL1}$,
$\mathrm{EDBL2}$,
$\mathrm{EDBL3}$
or
$\mathrm{EDBL6}$
(it is marked by a blue frame in Figure~\ref{fig:special_even_all_trees_2}).
Upon restoring $D(P)$
(attaching it to the central point of $R$)
we obtain a matching that fits the same structure, and therefore, is also special.
This is a contradiction as above.

\smallskip

\noindent\textbf{\boldmath Case 2.
$L$ is an $\mathrm{EDB1}$-matching with $j= \left \lceil \frac{\ell-1}{2} \right \rceil $.}
Refer to the labeling of $D(L)$ as in Figure~\ref{fig:special_even_all_trees_2}.
The proof is similar to that of Case~1.
If $D(K)$ is attached to $D(L)$
``in the right part'' -- that is,
in $B_i$ or $D_i$ with $i \geq j$, or in one of the points $E, E'$, --
we take $P$ to be the leftmost antiblock.
If $D(K)$ is attached to $D(L)$
``in the left part'' -- that is,
in $B_i$ or $D_i$ with $i < j$, --
we take $P$ to be the rightmost block.
We assume (for contradiction) that $N=M-P$ is special.
However, depending on the case,
$D(N)$ has an antiblock or a block with a vertex of degree $3$.
Therefore it can fit a special matching in a specific way.
Upon restoring $P$, we see that $D(M)$ fits the same structure as $D(N)$,
and, therefore, $M$ is special -- a contradiction.

\smallskip

\noindent\textbf{\boldmath Case~3.
$L$ is an $\mathrm{EDB}$-matching not of the kind treated in Case~2,
or an $\mathrm{EDBL}$-matching.}
By Corollary~\ref{thm:middle_even_struct},
$L$ is connected by a path to a matching $L'$ of the kind treated in Case~2.
Therefore, $M=L+K$
is connected by a path to $M'=L'+K'$
where $K'$ is either $K$ or its flip.
As we saw in Case~2, $M'$ has a desired decomposition,
therefore, the statement of Theorem holds for $M$.

\smallskip

We have verified all the cases, and, so, the proof is complete. \hfill\qed



\subsection{The order of the ring component}\label{sec:ring_big}
In Introduction, the ring component was referred to as the ``big component''.
In order to show that it indeed has the biggest order, we need to compare its order with that of medium components.

\begin{proposition}\label{thm:big}
For each $k \geq 9$,
the order of the ring component is larger than
the order of the components that contain DBD- (for odd $k$)
or, respectively EDB- (for even $k$)
matchings.
\end{proposition}

\begin{proof}
Since the total number of vertices in $\dcm_k$ is $C_k$,
and we know the order and the number of all other components,
we obtain that the for odd $k$ the order of the ring component of $\dcm_k$ is
\[ C_{2\ell-1} - 1\cdot \frac{1}{\ell}\binom{4\ell-2}{\ell-1} - \ell \cdot (2\ell-1)  2^{\ell-3},\] and for even $k$ it is
\[ C_{2\ell} - 2 \cdot \ell \, 2^{\ell-1} -  (6 \ell - 6) \cdot \ell \, 2^{\ell-2}.\]
Thus, we need to show that for odd $k \geq 9$ we have
\[C_{2\ell-1} - \frac{1}{\ell}\binom{4\ell-2}{\ell-1} - \ell (2\ell-1)  2^{\ell-3} > \ell,\]
or, equivalently,
\begin{equation}\label{eq:big_odd}
C_{2\ell-1} > \frac{1}{\ell}\binom{4\ell-2}{\ell-1} + \ell (2\ell-1)  2^{\ell-3} + \ell;
\end{equation}
and that for even $k \geq 10$ we have \[C_{2\ell} - \ell \, 2^{\ell} - \ell (6 \ell - 6) 2^{\ell-2} > 6\ell - 6,\]
or, equivalently,
\begin{equation}\label{eq:big_even}
C_{2\ell} > \ell \, 2^{\ell} + \ell (6 \ell - 6) 2^{\ell-2} + 6\ell - 6.
\end{equation}
First, notice that
Inequalities~\eqref{eq:big_odd} and~\eqref{eq:big_even} hold asymptotically since
the growth rate of
$(C_{2\ell-1})_{\ell\geq 1}$ and of $(C_{2\ell})_{\ell\geq 1}$ is $16$;
that of
$\left(\frac{1}{\ell}\binom{4\ell-2}{\ell-1}\right)_{\ell\geq 1}$ is $\frac{256}{27} \approx 9.48$;
and that of other terms is at most $2$.
In order to show that they hold for $k \geq 9$,
we verify them for $\ell=5$,
and show that for $\ell \geq 5$ 
we have
$\frac{\zrhs_{\ell+1}}{\zrhs_{\ell} } < 10$ and
$\frac{\zlhs_{\ell+1}}{\zlhs_{\ell} } > 10$ in them both.\footnote{
\textit{LHS} and \textit{RHS} denote 
the \textit{left-hand side} and the \textit{right-hand side}
of the respective inequalities.}
We omit further details.
\end{proof}

\section{More enumerating results, concluding remarks, and open problems}
\label{sec:conc}

\subsection{Vertices with largest degree}\label{sec:large_degree}
In Section~\ref{sec:small} we characterized matchings
with smallest possible degrees (as vertices of $\dcm_k$): $0$ and $1$.
One can expect that the matchings with the largest degree are the rings.
Here we show that this is indeed the case.

\begin{proposition}\label{thm:ring_degree}
For each $k>1$,
the vertices of $\dcm_k$ with the maximum degree are precisely those corresponding to the rings.
Their degree is the $k$th Riordan number,
\begin{equation}\label{eq:riordan}
  r_k = \frac{1}{k+1}\sum_{i=1}^{\left\lfloor \frac{k}{2} \right\rfloor} \binom{k+1}{i} \binom{k-i-1}{i-1}.
\end{equation}
\end{proposition}

\begin{proof}
Let $M$ be any matching of size $k$ which is not a ring.
Let $e = P_\alpha P_\beta$ be a diagonal edge of $M$.
Modify the point set $X_{2k}$ by transferring $P_\beta$
to the position between $P_\alpha$ and $P_{\alpha+1}$ (on $\mathbf{\Gamma}$).
Denote the modified point set by $X'_{2k}$
Let $M'$ be the matching of $X'_{2k}$ whose members connect the pairs of points with the same labels as $M$.
It is easy to see that $M'$ is a non-crossing matching,
and that each flippable partition of $M$
(given by labels of endpoints of edges)
is a flippable partition of $M'$.
Therefore $d(M) \leq d(M')$.
We repeat this procedure until we eventually reach a ring $R$.
Thus, we have $d(M) \leq d(R)$.
Moreover, since the partition that consists of one set
(whose members are all the edges)
is flippable in $R$ but not in $M$,
we have in fact $d(M) < d(R)$.

In order to find $d(R)$, we proceed as follows.
Assume that $R$ is the ring with edges $P_1P_2$, $P_3P_4, \dots$, $P_{2k-1}P_{2k}$.
For each $1 \leq i \leq k$, contract the edge $P_{2i-1}P_{2i}$ into the point $P_{2i}$.
The induced modification of flippable partitions of $R$
is a bijection 
between flippable partitions of $R$
and
non-crossing partitions of $\{Q_1, Q_2, \dots, Q_k\}$ without singletons.
The partitions of the latter type are known to be enumerated by Riordan numbers~\cite[A005043]{oeis}.
See~\cite{bernhart} for bijections between this structure and other structures enumerated by Riordan numbers.
The 
explicit formula for the $k$th Riordan numbers is as in Eq.~\eqref{eq:riordan}
(see~\cite{chen} for a simple combinatorial proof),
and asymptotically  $r_k=\Theta^*(3^k)$.
\end{proof}

\subsection{Number of edges}
In this section we consider enumeration of edges of $\dcm_k$.
Denote, for $k\geq 1$, the number of edges in $\dcm_k$ by $d_k$;
moreover, set $d_0=1$.
Let $z(x)$ be the corresponding generating function $\displaystyle{z(x) = \sum_{k\geq 0} d_k x^k}$,
and let $Z(x)=2z(x)-1$.

\begin{proposition}\label{thm:edges_gf}
The function $Z(x)$ satisfies the equation
\begin{equation}\label{eq:edges_gf}
Z(x) = 1+\frac{2x^2Z^4(x)}{1-xZ^2(x)}.
\end{equation}
Moreover, $d_k = \Theta^*(\mu^n)$ with $\mu \approx 5.27$.
\end{proposition}{

\begin{proof}
Any edge $e$ of $\dcm_{k}$ corresponds to a pair of disjoint compatible matchings -- say, $M_a$ and $M_b$.
By Observation~\ref{thm:def_cycles}, $M_a \cup M_b$ is a union of pairwise disjoint cycles
that consist alternatingly of edges of $M_a$ and $M_b$.
We can color them by blue and red, as in Figure~\ref{fig:chains_1}.
If we ignore the colors, these cycles form a non-crossing partition of $X_{2k}$ into even parts of size at least $4$.
Given such a partition, each polygon can be colored alternatingly by two colors in two ways.
Each way to color alternatingly all the polygons in such a partition
corresponds to an edge of $\dcm_k$. However, in this way each edge is created twice 
because exchanging all the colors results in the same edge.
Since each part in the partition can be colored in two ways, 
the total number of edges of $\dcm_{k}$
is equal to
the number of non-crossing partitions of $X_{2k}$ into even parts of size at least $4$,
when each partition is counted $2^{p-1}$ times,
where $p$ is the number of parts.
Equivalently, $H(x)$
is the generating function for the number of such partitions of $X_{2k}$
where each part is colored by one of two colors.
Since the part that contains $1$ is a polygon of even size at least $4$,
and the skip between any pair of consecutive points of this polygon
possibly contains further partition of the same kind,
we have
\[ 
Z(x) = 1 + 2x^2Z^4(x) + 2x^3Z^6(x) + 2x^4Z^8(x) + 2x^5Z^{10}(x) + \dots,
\] 
which is equivalent to Eq.~\eqref{eq:edges_gf}.

We can estimate the asymptotic growth rate of $(d_k)_{k \geq 0}$ as follows.
By the Exponential Growth Formula (see~\cite[IV.7]{flajolet}),
for an analytic function $f(x)$ the asymptotic growth rate 
is $\mu = \frac{1}{\lambda}$,
where $\lambda$ is the absolute value of the singularity of $f(x)$ closest to the origin.
It is easier to find $\lambda$ for $Y(x) = xZ(x)$.
From Eq.~\eqref{eq:edges_gf} we have
\[2Y^4(x) + Y^3(x) - x Y^2(x) -x Y(x) + x^2 = 0.\]
This is a square equation with respect to $x$; solving it we obtain that $Y(x)$ is the compositional inverse of
\[V(x)=\frac{x}{2}\left(1+x+\sqrt{1-2x-7x^2}\right).\] 
The singularity points of $Y(x)$ correspond to the points where the derivative of $V(x)$ vanishes.
Analyzing $V(x)$, we find that the singularity point of $Y(x)$ with 
the smallest absolute value is $\lambda \approx 0.1898$.
Therefore, the asymptotic growth rate of $(d_k)_{k \geq 0}$
is $\mu \approx 5.2680$.
\end{proof}

\subsection{``Almost perfect'' matchings for odd number of points}

In this section we consider, without going into details, the following variation.
Let $X_{2k+1}$ be a set of $2k+1$ points in convex position.
In this case we can speak about \textit{almost perfect} (non-crossing straight-line) matchings --
matchings of $2k$ out of these points, one point remaining unmatched.
Clearly, the number of such matchings is $kC_k$.
The definition of disjoint compatibility and that of disjoint compatibility graph 
are carried over for this case in a straightforward way.
In contrast to the case of perfect matchings of even number of points,
we have here the following result.

\begin{claim}
For each $k$, the disjoint compatibility graph 
of almost perfect matchings of $2k+1$ points in general position
is connected.
\end{claim}

This claim can be proven along the following lines.
For $k=1, 2$, it is verified directly. 
For $k \geq 3$, we apply induction similarly to that in the proof of Theorem~\ref{thm:ring_component}.
The \textit{rings} in this case are the matchings that contain only boundary edges and one unmatched point.
For fixed $k$, there are exactly $2k+1$ rings that are uniquely identified by their unmatched point.
Denote by $R_j$ the ring whose unmathced point is $P_j$.
Then the ring $R_j$ is disjoint compatible to exactly two rings, namely,
$R_{j-1}$ and $R_{j+1}$.
Thus, the rings induce a cycle of size $2k+1$.

Let $M$ be an almost perfect matching, and let $P$ be the unmatched point.
We show that $M$ is connected by a path to the rings as follows.
It is always possible to find a separated pair $K$ 
which is not interrupted by $P$ 
(suppose that $K$ connects the points $P_i, P_{i+1},P_{i+2},P_{i+3}$).
We let $K$ oscillate, while transforming $L=M-K$ into a ring $R$ (on $2k-3$ points).
It is possible to assume that after this process $K$ is replaced by an antiblock $K'$.
Now either $K'+R$ is a ring and we are done,
or $R$ has the edge $P_{i-1} P_{i+4}$.
In the latter case we continue the reconfiguration: 
$K'$ continues to oscillate, while we ``rotate'' $R$
so that its unmatched point moves clockwise.
Eventually, we will reach two matchings in which $R$ is replaced by 
rings whose unmatched points are $P_{i-1}$ and $P_{i+4}$.
For one of them, we still have the antiblock $K'$,
and the whole matching is a ring.

\subsection{Summary and open problems}
We showed that for sets of $2k$ points in convex position
the disjoint compatibility graph is always disconnected
(except for $k=1, 2$).
Moreover, we proved that for $k\geq 9$ there exist exactly three kinds of connected components:
small, medium and big.
For each $k$ we found the number of components of each kind.
For small and medium components, we determined precisely their structure.

For sets of points \textbf{in general position}, the disjoint compatibility graph depends on the order type.
Therefore only some questions concerning the structure can be asked in general.
We suggest the following problems for future research.

\begin{enumerate}
\item \textbf{Connectedness for a general point set.} 
What is more typical for set of points in general position:
being the disjoint compatibility graph connected or disconnected? 
The former possibility can be the case since, intuitively, one of the reasons 
for the disconnectedness when the points are in convex position
is the fact that all edges connect two points that lie on the boundary of the convex hull.
One can conjecture, for example, that 
the disjoint compatibility graph is connected
if the fraction of points in the interior of the convex hull is not too small.

\item \textbf{Isolated matchings.} 
In order to construct isolated matchings for sets of points not only in convex position, 
we can use the following recursive procedure. 
First, any matching of size $1$ is isolated.
Next, let $M=M_1 \cup \{e\} \cup M_2$, 
where $M_1$ and $M_2$ are isolated matchings,
and the edge $e$ blocks the visibility between $M_1$ and $M_2$ (see Figure~\ref{fig:iso}(a)).
Then it is easy to see that $M$ is also isolated.
For matchings of points in convex position, this construction gives all isolated matchings:
indeed, one can easily show that 
for this case this construction is equivalent to that 
from the definition of I-matchings (see Section~\ref{sec:small_odd}).
However, for points in general (not convex) position
it is possible to find an isolated matching that cannot be obtained by this recursive procedure:
see Figure~\ref{fig:iso}(b).

\begin{figure}[h]
$$\includegraphics[width=100mm]{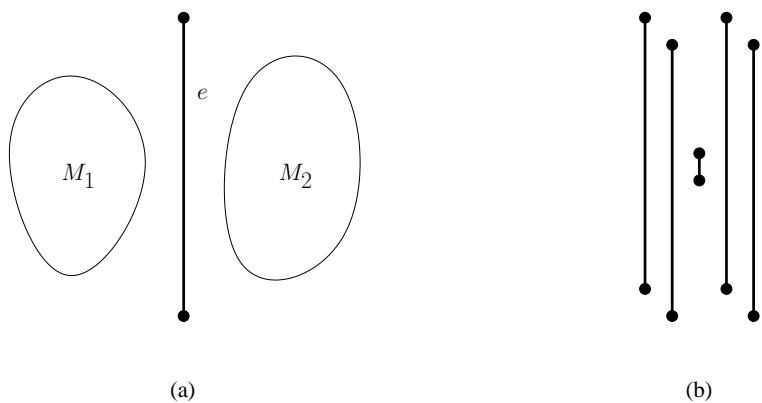}$$
\caption{(a) A recursive construction of isolated matchings. 
(b) An isolated matching that cannot be obtained by this construction.}
\label{fig:iso}
\end{figure}

\end{enumerate}

\section*{Acknowledgments}

Research on this paper was initiated while Tillmann Miltzow was
visiting the Institute for Software Technology of the University of Technology Graz.
We would like to thank
Thomas Hackl, Alexander Pilz and Birgit Vogtenhuber
for valuable discussions and comments.
Research of \mbox{Oswin} Aichholzer is supported by the ESF EUROCORES
programme EuroGIGA, CRP `ComPoSe', Austrian Science Fund (FWF): I648-N18.
Research of Andrei Asinowski is supported by the ESF EUROCORES programme EuroGIGA, CRP `ComPoSe',
Deutsche Forschungsgemeinschaft (DFG), grant FE 340/9-1.



\begin{thebibliography}{00}
\bibitem{aichholzer}
O.~Aichholzer, S.~Bereg, A.~Dumitrescu, A.~Garc\'ia, C.~Huemer, F.~Hurtado, M.~Kano, A.~M\'arquez,
D.~Rappaport, S.~Smorodinsky, D.~L.~Souvaine, J.~Urrutia, and D.~Wood.
Compatible geometric matchings.
\textit{Comput.\ Geom.}, 42 (2009), 617--626.

\bibitem{aah2002}
O.~Aichholzer, F.~Aurenhammer, and F.~Hurtado.
Sequences of spanning trees and a fixed tree theorem.
\textit{Comput.\ Geom.}, 21 (2002), 3--20.

\bibitem{abhpv2013}
O.~Aichholzer, L.~Barba, T.~Hackl, A.~Pilz, and B.~Vogtenhuber.
Linear transformation distance for bichromatic matchings.
\textit{Manuscript}, 2013;
arXiv:1312.0884 [cs.CG].

\bibitem {aght}
O.~Aichholzer, A.~Garc{\'i}a, F.~Hurtado, and J.~Tejel.
Compatible matchings in geometric graphs.
Proceedings of the XIV Encuentros de Geometr{\'i}a Computacional (2011),
145--148.

\bibitem{aloupis}
G.~Aloupis, L.~Barba, S.~Langerman, and D.~L.~Souvaine.
Bichromatic Compatible Matchings.
Proceedings of the 29th annual Symposuim on Computational Geometry (2013),
267--276.

\bibitem{br}
A.~Asinowski, T.~Miltzow, and G.~Rote.
Quasi-parallel segments and characterization of unique bichromatic matchings.
Proceedings of the 29th European Workshop on Computational Geometry (2013), 225--228.

\bibitem{bernhart}
F.~R.~Bernhart.
Catalan, Motzkin, and Riordan numbers.
\textit{Discrete Math.}, 204 (1999), 73--112.



\bibitem{brown}
W.~G.~Brown.
Historical note on a recurrent combinatorial problem.
\textit{Amer.\ Math.\ Monthly}, 72:9 (1965), 973--977.

\bibitem{cantini}
L.~Cantini and A.~Sportiello.
Proof of the Razumov-Stroganov conjecture.
\textit{J.\ Combin.\ Theory Ser.\ A},
118:5 (2011),
1549--1574.


\bibitem{chen}
W.~Y.~C.~Chen, E.~Y.~P.~Deng, and L.~L.~M.~Yang.
Riordan paths and derangements.
\textit{Discrete Math.}, 308:11 (2008), 2222--2227


\bibitem{finucan}
H.~M.~Finucan.
Some decompositions of generalized Catalan numbers.
Combinatorial Mathematics~IX.
Proceedings of the 9th Australian Conference on Combinatorial Mathematics (Brisbane, August 1981).
Ed.: E.\ J.\ Billington, S.\ Oates-Williams and A.\ P.\ Street.
Lecture Notes Math., 952. Springer-Verlag, 1982, 275--293.






\bibitem{flajolet}
P.~Flajolet and P.~Sedgewick.
Analytic Combinatorics.
Cambridge University Press, 2009.

\bibitem{francesco}
P.~Di Francesco, P.~Zinn-Justin, and J.-B.~Zuber.
A bijection between classes of Fully Packed Loops and plane partitions.
\textit{Electron.\ J.\ Combin.}, 11 \#R64 (2004).

\bibitem{fuss}
N.~Fuss.
Solutio quaestionis, quot modis polygonum $n$ laterum in polygona $m$ laterum per diagonales resolvi queat.
Nova Acta Academiae Scientiarum Imperialis Petropolitanae, IX (1791), 243--251.


\bibitem{garcia}
A.~Garc{\'i}a, M.~Noy, and J.~Tejel.
Lower bounds for the number of crossing-free subgraphs of~$K_n$.
\textit{Comput.\ Geom.}, 16 (2000), 211--221.



\bibitem{hernando}
C.~Hernando, F.~Hurtado, and M.~Noy.
Graphs of non-crossing perfect matchings.
\textit{Graphs and Combinatorics},
18 (2002), 517--532.



\bibitem{houle}
M.~E.~Houle, F.~Hurtado, M.~Noy, and E.~Rivera.
Graphs of triangulations and perfect matchings.
Graphs and Combinatorics, 21 (2005), 325--331.

\bibitem{Hurtado1999}
F.~Hurtado.
Flipping edges in triangulations.
\textit{Discrete Comput.\ Geom.}, 22:3 (1999), 333--346 .

\bibitem{ishaque}
M.~Ishaque, D.~L.~Souvaine and C.~D.~T\'oth.
Disjoint Compatible Geometric Matchings.
\textit{Discrete Comput.\ Geom.}, 49:1 (2013), 89--131.


\bibitem{propp}
J.~Propp.
The many faces of alternating-sign matrices.
In \textit{Discrete Models: Combinatorics, Computation, and Geometry},
Volume AA of DMTCS Proceedings (2001), 43--58.

\bibitem{Razen2008}
A.~Razen. Crossing-Free Configurations on Planar Point Sets.
Dissertation, ETH Zurich, No. 18607, 2009.
\href{http://dx.doi.org/10.3929/ethz-a-005902005}{http://dx.doi.org/10.3929/ethz-a-005902005}.
\textit{Also in:} A lower bound for the transformation of compatible perfect matchings.
Proceedings of the 24th European Workshop on Computational Geometry (2008), 115--118.

\bibitem{razumov}
A.~V.~Razumov and Yu.~G.~Stroganov.
Combinatorial nature of ground state vector of $O(1)$ loop model.
\textit{Theor.\ Math.\ Phys.}, 138:3 (2004), 333--337.


\bibitem{Sharir2006}
M.~Sharir and W.~Welzl.
On the number of crossing-free matchings, cycles, and partitions.
SIAM J.\ Comput.\ 36:3 (2006), 695--720.

\bibitem{stanley}
R.~P.~Stanley.
Enumerative Combinatorics.
Volume 2.
Cambridge University Press, 1999.

\bibitem{oeis}
\textit{The On-Line Encyclopedia of Integer Sequences},
\href{http://oeis.org/}{http://oeis.org/} .

\bibitem{temperley}
H.~N.~V.~Temperley and E.~H.~Lieb.
Relations between the `percolation' and `colouring'
problem and other graph-theoretical problems associated with
regular planar lattices: some exact results for
the `percolation' problem.
\textit{Proc.\ R.\ Soc.\ Lond. Ser.\ A Math.\ Phys.\ Sci.},
322:1549 (1971), 251--280.

\bibitem{wieland}
B.~Wieland.
Large dihedral Symmetry of the set of alternating sign matrices.
\textit{Electron.\ J.\ Combin.}, 7 \#R37 (2000).


\end{thebibliography}
\end{document}